\documentclass[11pt, reqno]{amsart}
\usepackage{amssymb}
\usepackage{eucal}
\usepackage{amsmath}
\usepackage{amscd}
\usepackage[dvips]{color}
\usepackage{multicol}
\usepackage[all]{xy}           
\usepackage{graphicx}
\usepackage{color}
\usepackage{colordvi}
\usepackage{xspace}
\usepackage{tikz}

\usepackage[active]{srcltx} 

\begin{document}
\newcommand {\emptycomment}[1]{} 

\baselineskip=14pt
\newcommand{\nc}{\newcommand}
\newcommand{\delete}[1]{}
\nc{\mfootnote}[1]{\footnote{#1}} 
\nc{\todo}[1]{\tred{To do:} #1}

\nc{\mlabel}[1]{\label{#1}}  
\nc{\mcite}[1]{\cite{#1}}  
\nc{\mref}[1]{\ref{#1}}  
\nc{\mbibitem}[1]{\bibitem{#1}} 

\delete{
\nc{\mlabel}[1]{\label{#1}  
{\hfill \hspace{1cm}{\bf{{\ }\hfill(#1)}}}}
\nc{\mcite}[1]{\cite{#1}{{\bf{{\ }(#1)}}}}  
\nc{\mref}[1]{\ref{#1}{{\bf{{\ }(#1)}}}}  
\nc{\mbibitem}[1]{\bibitem[\bf #1]{#1}} 
}

\newtheorem{thm}{Theorem}[section]
\newtheorem{lem}[thm]{Lemma}
\newtheorem{cor}[thm]{Corollary}
\newtheorem{pro}[thm]{Proposition}
\newtheorem{ex}[thm]{Example}
\newtheorem{rmk}[thm]{Remark}
\newtheorem{defi}[thm]{Definition}
\newtheorem{pdef}[thm]{Proposition-Definition}
\newtheorem{condition}[thm]{Condition}

\renewcommand{\labelenumi}{{\rm(\alph{enumi})}}
\renewcommand{\theenumi}{\alph{enumi}}

\nc{\tred}[1]{\textcolor{red}{#1}}
\nc{\tblue}[1]{\textcolor{blue}{#1}}
\nc{\tgreen}[1]{\textcolor{green}{#1}}
\nc{\tpurple}[1]{\textcolor{purple}{#1}}
\nc{\btred}[1]{\textcolor{red}{\bf #1}}
\nc{\btblue}[1]{\textcolor{blue}{\bf #1}}
\nc{\btgreen}[1]{\textcolor{green}{\bf #1}}
\nc{\btpurple}[1]{\textcolor{purple}{\bf #1}}

\nc{\cm}[1]{\textcolor{red}{Chengming:#1}}
\nc{\yy}[1]{\textcolor{blue}{Yanyong: #1}}
\nc{\lit}[2]{\textcolor{blue}{#1}{}} 
\nc{\yh}[1]{\textcolor{green}{Yunhe: #1}}


\nc{\twovec}[2]{\left(\begin{array}{c} #1 \\ #2\end{array} \right )}
\nc{\threevec}[3]{\left(\begin{array}{c} #1 \\ #2 \\ #3 \end{array}\right )}
\nc{\twomatrix}[4]{\left(\begin{array}{cc} #1 & #2\\ #3 & #4 \end{array} \right)}
\nc{\threematrix}[9]{{\left(\begin{matrix} #1 & #2 & #3\\ #4 & #5 & #6 \\ #7 & #8 & #9 \end{matrix} \right)}}
\nc{\twodet}[4]{\left|\begin{array}{cc} #1 & #2\\ #3 & #4 \end{array} \right|}

\nc{\rk}{\mathrm{r}}
\newcommand{\g}{\mathfrak g}
\newcommand{\h}{\mathfrak h}
\newcommand{\pf}{\noindent{$Proof$.}\ }
\newcommand{\frkg}{\mathfrak g}
\newcommand{\frkh}{\mathfrak h}
\newcommand{\Id}{\rm{Id}}
\newcommand{\gl}{\mathfrak {gl}}
\newcommand{\ad}{\mathrm{ad}}
\newcommand{\add}{\frka\frkd}
\newcommand{\frka}{\mathfrak a}
\newcommand{\frkb}{\mathfrak b}
\newcommand{\frkc}{\mathfrak c}
\newcommand{\frkd}{\mathfrak d}
\newcommand {\comment}[1]{{\marginpar{*}\scriptsize\textbf{Comments:} #1}}

\nc{\gensp}{V} 
\nc{\relsp}{\Lambda} 
\nc{\leafsp}{X}    
\nc{\treesp}{\overline{\calt}} 

\nc{\vin}{{\mathrm Vin}}    
\nc{\lin}{{\mathrm Lin}}    

\nc{\gop}{{\,\omega\,}}     
\nc{\gopb}{{\,\nu\,}}
\nc{\svec}[2]{{\tiny\left(\begin{matrix}#1\\
#2\end{matrix}\right)\,}}  
\nc{\ssvec}[2]{{\tiny\left(\begin{matrix}#1\\
#2\end{matrix}\right)\,}} 

\nc{\typeI}{local cocycle $3$-Lie bialgebra\xspace}
\nc{\typeIs}{local cocycle $3$-Lie bialgebras\xspace}
\nc{\typeII}{double construction $3$-Lie bialgebra\xspace}
\nc{\typeIIs}{double construction $3$-Lie bialgebras\xspace}

\nc{\bia}{{$\mathcal{P}$-bimodule ${\bf k}$-algebra}\xspace}
\nc{\bias}{{$\mathcal{P}$-bimodule ${\bf k}$-algebras}\xspace}

\nc{\rmi}{{\mathrm{I}}}
\nc{\rmii}{{\mathrm{II}}}
\nc{\rmiii}{{\mathrm{III}}}
\nc{\pr}{{\mathrm{pr}}}
\newcommand{\huaA}{\mathcal{A}}

\nc{\pll}{\beta}
\nc{\plc}{\epsilon}

\nc{\ass}{{\mathit{Ass}}}
\nc{\lie}{{\mathit{Lie}}}
\nc{\comm}{{\mathit{Comm}}}
\nc{\dend}{{\mathit{Dend}}}
\nc{\zinb}{{\mathit{Zinb}}}
\nc{\tdend}{{\mathit{TDend}}}
\nc{\prelie}{{\mathit{preLie}}}
\nc{\postlie}{{\mathit{PostLie}}}
\nc{\quado}{{\mathit{Quad}}}
\nc{\octo}{{\mathit{Octo}}}
\nc{\ldend}{{\mathit{ldend}}}
\nc{\lquad}{{\mathit{LQuad}}}

 \nc{\adec}{\check{;}} \nc{\aop}{\alpha}
\nc{\dftimes}{\widetilde{\otimes}} \nc{\dfl}{\succ} \nc{\dfr}{\prec}
\nc{\dfc}{\circ} \nc{\dfb}{\bullet} \nc{\dft}{\star}
\nc{\dfcf}{{\mathbf k}} \nc{\apr}{\ast} \nc{\spr}{\cdot}
\nc{\twopr}{\circ} \nc{\tspr}{\star} \nc{\sempr}{\ast}
\nc{\disp}[1]{\displaystyle{#1}}
\nc{\bin}[2]{ (_{\stackrel{\scs{#1}}{\scs{#2}}})}  
\nc{\binc}[2]{ \left (\!\! \begin{array}{c} \scs{#1}\\
    \scs{#2} \end{array}\!\! \right )}  
\nc{\bincc}[2]{  \left ( {\scs{#1} \atop
    \vspace{-.5cm}\scs{#2}} \right )}  
\nc{\sarray}[2]{\begin{array}{c}#1 \vspace{.1cm}\\ \hline
    \vspace{-.35cm} \\ #2 \end{array}}
\nc{\bs}{\bar{S}} \nc{\dcup}{\stackrel{\bullet}{\cup}}
\nc{\dbigcup}{\stackrel{\bullet}{\bigcup}} \nc{\etree}{\big |}
\nc{\la}{\longrightarrow} \nc{\fe}{\'{e}} \nc{\rar}{\rightarrow}
\nc{\dar}{\downarrow} \nc{\dap}[1]{\downarrow
\rlap{$\scriptstyle{#1}$}} \nc{\uap}[1]{\uparrow
\rlap{$\scriptstyle{#1}$}} \nc{\defeq}{\stackrel{\rm def}{=}}
\nc{\dis}[1]{\displaystyle{#1}} \nc{\dotcup}{\,
\displaystyle{\bigcup^\bullet}\ } \nc{\sdotcup}{\tiny{
\displaystyle{\bigcup^\bullet}\ }} \nc{\hcm}{\ \hat{,}\ }
\nc{\hcirc}{\hat{\circ}} \nc{\hts}{\hat{\shpr}}
\nc{\lts}{\stackrel{\leftarrow}{\shpr}}
\nc{\rts}{\stackrel{\rightarrow}{\shpr}} \nc{\lleft}{[}
\nc{\lright}{]} \nc{\uni}[1]{\tilde{#1}} \nc{\wor}[1]{\check{#1}}
\nc{\free}[1]{\bar{#1}} \nc{\den}[1]{\check{#1}} \nc{\lrpa}{\wr}
\nc{\curlyl}{\left \{ \begin{array}{c} {} \\ {} \end{array}
    \right .  \!\!\!\!\!\!\!}
\nc{\curlyr}{ \!\!\!\!\!\!\!
    \left . \begin{array}{c} {} \\ {} \end{array}
    \right \} }
\nc{\leaf}{\ell}       
\nc{\longmid}{\left | \begin{array}{c} {} \\ {} \end{array}
    \right . \!\!\!\!\!\!\!}
\nc{\ot}{\otimes} \nc{\sot}{{\scriptstyle{\ot}}}
\nc{\otm}{\overline{\ot}}
\nc{\ora}[1]{\stackrel{#1}{\rar}}
\nc{\ola}[1]{\stackrel{#1}{\la}}
\nc{\pltree}{\calt^\pl}
\nc{\epltree}{\calt^{\pl,\NC}}
\nc{\rbpltree}{\calt^r}
\nc{\scs}[1]{\scriptstyle{#1}} \nc{\mrm}[1]{{\rm #1}}
\nc{\dirlim}{\displaystyle{\lim_{\longrightarrow}}\,}
\nc{\invlim}{\displaystyle{\lim_{\longleftarrow}}\,}
\nc{\mvp}{\vspace{0.5cm}} \nc{\svp}{\vspace{2cm}}
\nc{\vp}{\vspace{8cm}} \nc{\proofbegin}{\noindent{\bf Proof: }}
\nc{\proofend}{$\blacksquare$ \vspace{0.5cm}}
\nc{\freerbpl}{{F^{\mathrm RBPL}}}
\nc{\sha}{{\mbox{\cyr X}}}  
\nc{\ncsha}{{\mbox{\cyr X}^{\mathrm NC}}} \nc{\ncshao}{{\mbox{\cyr
X}^{\mathrm NC,\,0}}}
\nc{\shpr}{\diamond}    
\nc{\shprm}{\overline{\diamond}}    
\nc{\shpro}{\diamond^0}    
\nc{\shprr}{\diamond^r}     
\nc{\shpra}{\overline{\diamond}^r}
\nc{\shpru}{\check{\diamond}} \nc{\catpr}{\diamond_l}
\nc{\rcatpr}{\diamond_r} \nc{\lapr}{\diamond_a}
\nc{\sqcupm}{\ot}
\nc{\lepr}{\diamond_e} \nc{\vep}{\varepsilon} \nc{\labs}{\mid\!}
\nc{\rabs}{\!\mid} \nc{\hsha}{\widehat{\sha}}
\nc{\lsha}{\stackrel{\leftarrow}{\sha}}
\nc{\rsha}{\stackrel{\rightarrow}{\sha}} \nc{\lc}{\lfloor}
\nc{\rc}{\rfloor}
\nc{\tpr}{\sqcup}
\nc{\nctpr}{\vee}
\nc{\plpr}{\star}
\nc{\rbplpr}{\bar{\plpr}}
\nc{\sqmon}[1]{\langle #1\rangle}
\nc{\forest}{\calf}
\nc{\altx}{\Lambda_X} \nc{\vecT}{\vec{T}} \nc{\onetree}{\bullet}
\nc{\Ao}{\check{A}}
\nc{\seta}{\underline{\Ao}}
\nc{\deltaa}{\overline{\delta}}
\nc{\trho}{\tilde{\rho}}

\nc{\rpr}{\circ}
\nc{\dpr}{{\tiny\diamond}}
\nc{\rprpm}{{\rpr}}

\nc{\mmbox}[1]{\mbox{\ #1\ }} \nc{\ann}{\mrm{ann}}
\nc{\Aut}{\mrm{Aut}} \nc{\can}{\mrm{can}}
\nc{\twoalg}{{two-sided algebra}\xspace}
\nc{\colim}{\mrm{colim}}
\nc{\Cont}{\mrm{Cont}} \nc{\rchar}{\mrm{char}}
\nc{\cok}{\mrm{coker}} \nc{\dtf}{{R-{\rm tf}}} \nc{\dtor}{{R-{\rm
tor}}}
\renewcommand{\det}{\mrm{det}}
\nc{\depth}{{\mrm d}}
\nc{\Div}{{\mrm Div}} \nc{\End}{\mrm{End}} \nc{\Ext}{\mrm{Ext}}
\nc{\Fil}{\mrm{Fil}} \nc{\Frob}{\mrm{Frob}} \nc{\Gal}{\mrm{Gal}}
\nc{\GL}{\mrm{GL}} \nc{\Hom}{\mrm{Hom}} \nc{\hsr}{\mrm{H}}
\nc{\hpol}{\mrm{HP}} \nc{\id}{\mrm{id}} \nc{\im}{\mrm{im}}
\nc{\incl}{\mrm{incl}} \nc{\length}{\mrm{length}}
\nc{\LR}{\mrm{LR}} \nc{\mchar}{\rm char} \nc{\NC}{\mrm{NC}}
\nc{\mpart}{\mrm{part}} \nc{\pl}{\mrm{PL}}
\nc{\ql}{{\QQ_\ell}} \nc{\qp}{{\QQ_p}}
\nc{\rank}{\mrm{rank}} \nc{\rba}{\rm{RBA }} \nc{\rbas}{\rm{RBAs }}
\nc{\rbpl}{\mrm{RBPL}}
\nc{\rbw}{\rm{RBW }} \nc{\rbws}{\rm{RBWs }} \nc{\rcot}{\mrm{cot}}
\nc{\rest}{\rm{controlled}\xspace}
\nc{\rdef}{\mrm{def}} \nc{\rdiv}{{\rm div}} \nc{\rtf}{{\rm tf}}
\nc{\rtor}{{\rm tor}} \nc{\res}{\mrm{res}} \nc{\SL}{\mrm{SL}}
\nc{\Spec}{\mrm{Spec}} \nc{\tor}{\mrm{tor}} \nc{\Tr}{\mrm{Tr}}
\nc{\mtr}{\mrm{sk}}

\nc{\ab}{\mathbf{Ab}} \nc{\Alg}{\mathbf{Alg}}
\nc{\Algo}{\mathbf{Alg}^0} \nc{\Bax}{\mathbf{Bax}}
\nc{\Baxo}{\mathbf{Bax}^0} \nc{\RB}{\mathbf{RB}}
\nc{\RBo}{\mathbf{RB}^0} \nc{\BRB}{\mathbf{RB}}
\nc{\Dend}{\mathbf{DD}} \nc{\bfk}{{\bf k}} \nc{\bfone}{{\bf 1}}
\nc{\base}[1]{{a_{#1}}} \nc{\detail}{\marginpar{\bf More detail}
    \noindent{\bf Need more detail!}
    \svp}
\nc{\Diff}{\mathbf{Diff}} \nc{\gap}{\marginpar{\bf
Incomplete}\noindent{\bf Incomplete!!}
    \svp}
\nc{\FMod}{\mathbf{FMod}} \nc{\mset}{\mathbf{MSet}}
\nc{\rb}{\mathrm{RB}} \nc{\Int}{\mathbf{Int}}
\nc{\Mon}{\mathbf{Mon}}
\nc{\remarks}{\noindent{\bf Remarks: }}
\nc{\OS}{\mathbf{OS}} 
\nc{\Rep}{\mathbf{Rep}}
\nc{\Rings}{\mathbf{Rings}} \nc{\Sets}{\mathbf{Sets}}
\nc{\DT}{\mathbf{DT}}

\nc{\BA}{{\mathbb A}} \nc{\CC}{{\mathbb C}} \nc{\DD}{{\mathbb D}}
\nc{\EE}{{\mathbb E}} \nc{\FF}{{\mathbb F}} \nc{\GG}{{\mathbb G}}
\nc{\HH}{{\mathbb H}} \nc{\LL}{{\mathbb L}} \nc{\NN}{{\mathbb N}}
\nc{\QQ}{{\mathbb Q}} \nc{\RR}{{\mathbb R}} \nc{\BS}{{\mathbb{S}}} \nc{\TT}{{\mathbb T}}
\nc{\VV}{{\mathbb V}} \nc{\ZZ}{{\mathbb Z}}


\nc{\calao}{{\mathcal A}} \nc{\cala}{{\mathcal A}}
\nc{\calc}{{\mathcal C}} \nc{\cald}{{\mathcal D}}
\nc{\cale}{{\mathcal E}} \nc{\calf}{{\mathcal F}}
\nc{\calfr}{{{\mathcal F}^{\,r}}} \nc{\calfo}{{\mathcal F}^0}
\nc{\calfro}{{\mathcal F}^{\,r,0}} \nc{\oF}{\overline{F}}
\nc{\calg}{{\mathcal G}} \nc{\calh}{{\mathcal H}}
\nc{\cali}{{\mathcal I}} \nc{\calj}{{\mathcal J}}
\nc{\call}{{\mathcal L}} \nc{\calm}{{\mathcal M}}
\nc{\caln}{{\mathcal N}} \nc{\calo}{{\mathcal O}}
\nc{\calp}{{\mathcal P}} \nc{\calq}{{\mathcal Q}} \nc{\calr}{{\mathcal R}}
\nc{\calt}{{\mathcal T}} \nc{\caltr}{{\mathcal T}^{\,r}}
\nc{\calu}{{\mathcal U}} \nc{\calv}{{\mathcal V}}
\nc{\calw}{{\mathcal W}} \nc{\calx}{{\mathcal X}}
\nc{\CA}{\mathcal{A}}

\nc{\fraka}{{\mathfrak a}} \nc{\frakB}{{\mathfrak B}}
\nc{\frakb}{{\mathfrak b}} \nc{\frakd}{{\mathfrak d}}
\nc{\oD}{\overline{D}}
\nc{\frakF}{{\mathfrak F}} \nc{\frakg}{{\mathfrak g}}
\nc{\frakm}{{\mathfrak m}} \nc{\frakM}{{\mathfrak M}}
\nc{\frakMo}{{\mathfrak M}^0} \nc{\frakp}{{\mathfrak p}}
\nc{\frakS}{{\mathfrak S}} \nc{\frakSo}{{\mathfrak S}^0}
\nc{\fraks}{{\mathfrak s}} \nc{\os}{\overline{\fraks}}
\nc{\frakT}{{\mathfrak T}}
\nc{\oT}{\overline{T}}
\nc{\frakX}{{\mathfrak X}} \nc{\frakXo}{{\mathfrak X}^0}
\nc{\frakx}{{\mathbf x}}
\nc{\frakTx}{\frakT}      
\nc{\frakTa}{\frakT^a}        
\nc{\frakTxo}{\frakTx^0}   
\nc{\caltao}{\calt^{a,0}}   
\nc{\ox}{\overline{\frakx}} \nc{\fraky}{{\mathfrak y}}
\nc{\frakz}{{\mathfrak z}} \nc{\oX}{\overline{X}}

\font\cyr=wncyr10

\nc{\redtext}[1]{\textcolor{red}{#1}}


\title{On antisymmetric infinitesimal conformal bialgebras}

\author{Yanyong Hong}
\address{Department of Mathematics, Hangzhou Normal University,
Hangzhou 311121, PR China}
\email{yyhong@hznu.edu.cn}

\author{Chengming Bai}
\address{Chern Institute of Mathematics \& LPMC, Nankai University, Tianjin 300071, PR China}
\email{baicm@nankai.edu.cn}

\subjclass[2010]{16D20, 16D70, 17A30, 17B38} \keywords{associative
conformal algebra, dendriform conformal algebra, associative
conformal Yang-Baxter equation, $\mathcal{O}$-operator,
Rota-Baxter operator}

\begin{abstract}
In this paper, we construct a bialgebra theory for associative
conformal algebras, namely antisymmetric infinitesimal conformal
bialgebras. On the one hand, it is an attempt to give conformal
structures for antisymmetric infinitesimal bialgebras. On the
other hand, under certain conditions, such structures are
equivalent to double constructions of Frobenius conformal
algebras, which are associative conformal algebras that are
decomposed into the direct sum of another associative conformal
algebra and its conformal dual as $\mathbb{C}[\partial]$-modules
such that both of them are subalgebras and the natural conformal
bilinear form is invariant. The coboundary case leads to the
introduction of associative conformal Yang-Baxter equation whose
antisymmetric solutions give antisymmetric infinitesimal conformal
bialgebras. Moreover, the construction of antisymmetric solutions
of associative conformal Yang-Baxter equation is given from
$\mathcal{O}$-operators of associative conformal algebras as well
as dendriform conformal algebras.
\end{abstract}

\maketitle

\section{Introduction}
The theory of Lie conformal algebras appeared as a formal language
describing the algebraic properties of the operator product
expansion in two-dimensional conformal field theory (\cite{K1}).
In particular, Lie conformal algebras turn out to be valuable
tools in studying vertex algebras and Hamiltonian formalism in the
theory of nonlinear evolution equations (\cite{BDK}).
Moreover, Lie conformal algebras have close connections to
infinite-dimensional Lie algebras satisfying the locality property
(\cite{K}). The conformal analogues of associative algebras,
namely, associative conformal algebras naturally appeared in the
representation theory of Lie conformal algebras (\cite{DK1}). They were studied widely on the structure theory
(\cite{BFK1,BFK2,BFK3,BKV,D,H1,K2,Ko1,Ko3,R1,R2,Ro1,Ro2,Ro3,Z1,Z2}) as well as
representation theory (\cite{BKL,Ko2,Ko4}). We
would like to point that there are the ``conformal analogues" for
certain algebras besides Lie and associative algebras or the
``conformal structures"
 of these algebras such as left-symmetric conformal
algebras (\cite{HL1}) and Jordan conformal algebras (\cite{KA}).

It is natural to extend such structures to the bialgebras, that
is, consider the conformal analogues of bialgebras. In the case of
Lie bialgebras, Liberati in \cite{L} developed a theory of Lie
conformal bialgebras including the introduction of the notions of
conformal classical Yang-Baxter equation, conformal Manin triples
and conformal Drinfeld's doubles. Similarly, a theory of
left-symmetric conformal bialgebras was developed in \cite{HL},
which are equivalent to a class of special Lie conformal algebras
named parak\"ahler Lie conformal algebras and the notion of
conformal $S$-equation was introduced in the coboundary case.
Moreover, the operator forms of the conformal classical
Yang-Baxter equation and the conformal $S$-equation were studied
in \cite{HB}, which shows that the antisymmetric solutions of the
conformal classical Yang-Baxter equation  and the symmetric
solutions of the conformal $S$-equation can be interpreted in
terms of a kind of operators called $\mathcal{O}$-operators in the
conformal sense.

But as far as we know, there is not a conformal analogue of
``associative bialgebras" yet. In fact, there are two kinds of
``associative bialgebras". One is the usual bialgebras in the
theory of Hopf algebras, which the coproducts are homomorphisms of
the products. Another is the infinitesimal bialgebras, which the
coproducts are ``derivations" of the products in certain sense,
introduced by Joni and Rota in \cite{JR} in order to provide an
algebraic framework for the calculus of divided difference. In
particular, for the latter, in the case of antisymmetric
infinitesimal (ASI) bialgebras which are called ``associative
D-algebras" in \cite{Zhe} or ``balanced infinitesimal bialgebras"
in the sense of the opposite algebras in \cite{A1}, there is a
systematic study in terms of their equivalences with double
constructions of Frobenius algebras as well as their relationships
with associative Yang-Baxter equation (\cite{Bai1}).



In this paper, we develop  a ``conformal" theory for antisymmetric infinitesimal bialgebras, namely antisymmetric infinitesimal (ASI) conformal
bialgebras. It is also a bialgebra theory for associative
conformal algebras. That is, the following diagram is commutative:

\bigskip
\begin{center}
\unitlength 1mm 
\linethickness{0.4pt}
\ifx\plotpoint\undefined\newsavebox{\plotpoint}\fi 
\begin{picture}(83.5,34)(0,0)
\put(5,6){associative algebras}
\put(72.5,6.5){associative conformal algebras}
\put(6.5,31.75){ASI bialgebras}
\put(71.75,30.75){ASI conformal bialgebras}
\put(83,12.5){\vector(0,1){13.5}}
\put(18,12.5){\vector(0,1){13.5}}
\put(38.75,34){conformal structures}
\put(38.75,8.75){conformal structures}
\put(83.5,18.5){bialgebra structures }
\put(18.5,18.75){bialgebra structures}
\put(38.5,32){\vector(1,0){32.25}}
\put(38.5,6.75){\vector(1,0){32.25}}
\end{picture}\qquad \qquad \qquad\qquad \qquad \qquad\qquad
\end{center}
We would like to point out that such an approach might not be
available for considering the conformal analogues of the usual
(associative) bialgebras and even Hopf algebras, but can help to
shed light on further studies on the latter.

The main idea is to extend the study of ASI bialgebras given in
\cite{Bai1} to the ``conformal case". Explicitly, we first
introduce the notion of double constructions of Frobenius
conformal algebras as  a conformal analogue of double
constructions of Frobenius algebras, which are associative
conformal algebras that are decomposed into the direct sum of
another associative conformal algebra and its conformal dual as a
$\mathbb{C}[\partial]$-module such that both of them are
subalgebras and the natural conformal bilinear form is invariant.
Such structures are interpreted equivalently in terms of matched
pairs of associative conformal algebras which were introduced in
\cite{H1}. Finally the notion of antisymmetric infinitesimal (ASI)
conformal bialgebras is introduced as equivalent structures of the
aforementioned matched pairs of associative conformal algebras as
well as the double constructions of Frobenius conformal algebras.
Note that the notion of ASI conformal bialgebras is available for
any associative conformal algebra, whereas the equivalence with
double constructions of Frobenius conformal algebras is available
for the associative conformal algebras which are finitely
generated and free as $\mathbb{C}[\partial]$-modules.

As in the case of ASI bialgebras, the definition of  coboundary
ASI conformal bialgebra is introduced and its study 
is  also meaningful.  It leads to the introduction of
associative conformal Yang-Baxter equation as a conformal analogue
of the associative Yang-Baxter equation. In particular,  its
antisymmetric solutions give ASI conformal bialgebras. The
associative conformal Yang-Baxter equation is interpreted in terms
of its operator forms by introducing the notion of
$\mathcal{O}$-operators of associative conformal algebras,
especially an antisymmetric solution of the associative conformal
Yang-Baxter equation corresponds to the skew-symmetric part of a
conformal linear map $T$, where $T_0=T_\lambda\mid_{\lambda=0}$ is
an $\mathcal{O}$-operator in the conformal sense and moreover, an
$\mathcal O$-operator of an associative conformal algebra gives an
antisymmetric solution of associative conformal Yang-Baxter
equation in a semi-direct product associative conformal algebra.
Furthermore, we introduce the notion of dendriform conformal
algebras and show that for a dendriform conformal algebra which is
finite and free as a $\mathbb{C}[\partial]$-module, there is a
natural $\mathcal O$-operator on the associated conformal
associative algebra. Therefore there are constructions of
antisymmetric solutions of associative conformal Yang-Baxter
equation and hence ASI conformal bialgebras from
$\mathcal{O}$-operators of associative conformal algebras as well
as dendriform conformal algebras.

This paper is organized as follows. In Section 2, the notions of
an associative conformal algebra, its bimodule and a matched pair
of associative conformal algebras are recalled. 
In Section 3, we introduce the notion of double constructions
of Frobenius conformal algebras and study  their relationship with
matched pairs of associative conformal algebras. In Section 4, the
notion of ASI conformal bialgebras is introduced  as (under
certain conditions) equivalent structures of the aforementioned
matched pair of associative conformal algebras as well as the
double constructions of Frobenius conformal algebras. Section 5 is
devoted to studying the coboundary case of ASI conformal bialgebras
and the associative conformal Yang-Baxter equation is introduced.
In Section 6, we introduce the notions of $\mathcal O$-operators
of associative conformal algebras and dendriform conformal
algebras to construct (antisymmetric) solutions of associative
conformal Yang-Baxter equation and hence give ASI conformal
bialgebras.

Throughout this paper, we denote by $\mathbb{C}$ the field of complex
numbers. All tensors over $\mathbb{C}$ are denoted by $\otimes$. We denote the identity map by $I$.
Moreover, if $A$ is a vector space, then the space of polynomials of $\lambda$ with coefficients in $A$ is denoted by $A[\lambda]$.

\section{Preliminaries on  associative conformal algebras}
We recall the notions of an associative conformal algebra, its
bimodule and  a matched pair of associative conformal algebras.
The interested readers may consult  \cite{K1} and \cite{H1} for
more details. 

\begin{defi}\label{def1}{\rm
A {\bf conformal algebra} $R$ is a $\mathbb{C}[\partial]$-module
endowed with a $\mathbb{C}$-bilinear map $\cdot_\lambda \cdot: R\times R\rightarrow
 R[\lambda]$, $(a, b)\mapsto
a_{\lambda} b$ satisfying
\begin{eqnarray}
(\partial a)_{\lambda}b=-\lambda a_{\lambda}b,  \quad
a_{\lambda}(\partial b)=(\partial+\lambda)a_{\lambda}b,
\quad\forall\ a,b\in R.~~~~~~\quad\text{(conformal sesquilinearity)}
\end{eqnarray}

An {\bf associative conformal algebra} $R$ is a conformal algebra with the $\mathbb{C}$-bilinear
map $\cdot_\lambda \cdot: R\times R\rightarrow R[\lambda]$ satisfying
\begin{eqnarray}
(a_{\lambda}b)_{\lambda+\mu}c=a_{\lambda}(b_\mu c), \quad \forall\
a,b,c\in R.
\end{eqnarray}

}
\end{defi}

A  conformal algebra is called {\bf finite} if it is finitely
generated as a $\mathbb{C}[\partial]$-module. The {\bf rank} of a
conformal algebra $R$ is its rank as a
$\mathbb{C}[\partial]$-module. The notions of a homomorphism, an
ideal and a subalgebra of an associative conformal algebra are
defined as usual.


\begin{ex}
Let $(A,\cdot)$ be an associative algebra. Then
$\text{Cur}(A)=\mathbb{C}[\partial]\otimes A$ is an associative
conformal algebra with the following $\lambda$-product:
\begin{eqnarray*}
(p(\partial)a)_\lambda (q(\partial)b)=p(-\lambda)q(\lambda+\partial)
(a\cdot b), \;\;\forall\ \text{$p(\partial)$, $q(\partial)\in
\mathbb{C}[\partial]$, $a$, $b\in A$.}
\end{eqnarray*}
\end{ex}

\begin{defi}\label{deff1}{\rm
A {\bf left module} $M$ over an associative conformal algebra $A$ is a $\mathbb{C}[\partial]$-module endowed with a $\mathbb{C}$-bilinear map
$A\times M\longrightarrow M[\lambda]$, $(a, v)\mapsto a\rightharpoonup_\lambda v$, satisfying the following axioms $(\forall\ a, b\in A, v\in M)$:\\
(LM1)$\qquad\qquad (\partial a)\rightharpoonup_\lambda v=-\lambda a\rightharpoonup_\lambda v,~~~a\rightharpoonup_\lambda(\partial v)=(\partial+\lambda)(a\rightharpoonup_\lambda v),$\\
(LM2)$\qquad\qquad (a_\lambda b)\rightharpoonup_{\lambda+\mu}v=a\rightharpoonup_\lambda(b\rightharpoonup_\mu v).$\\
We denote it by $(M,\rightharpoonup_\lambda)$.

A {\bf right module} $M$ over an associative conformal algebra $A$ is a $\mathbb{C}[\partial]$-module endowed with a $\mathbb{C}$-bilinear map
$M\times A\longrightarrow M[\lambda]$, $(v, a)\mapsto v\leftharpoonup_\lambda a$, satisfying the following axioms $(\forall\ a, b\in A, v\in M)$:\\
(RM1)$\qquad\qquad (\partial v)\leftharpoonup_\lambda a=-\lambda v\leftharpoonup_\lambda a,~~~v\leftharpoonup_\lambda(\partial a)=(\partial+\lambda)(v\leftharpoonup_\lambda a),$\\
(RM2)$\qquad\qquad (v\leftharpoonup_\lambda a)\leftharpoonup_{\lambda+\mu}b=v\leftharpoonup_\lambda(a_\mu b).$\\
We denote it by $(M,\leftharpoonup_\lambda)$.

An {\bf $A$-bimodule} is a triple
$(M,\rightharpoonup_\lambda,\leftharpoonup_\lambda)$ such that
$(M,\rightharpoonup_\lambda)$ is a left $A$-module,
$(M,\leftharpoonup_\lambda)$ is a right $A$-module, and they satisfy the
following condition
\begin{eqnarray}
(a\rightharpoonup_\lambda
v)\leftharpoonup_{\lambda+\mu}b=a\rightharpoonup_\lambda(v\leftharpoonup_\mu
b),\;\;\forall\ a,b\in A, v\in M.
\end{eqnarray}
}\end{defi}

\begin{defi}{\rm
Let $U$ and $V$ be two $\mathbb{C}[\partial]$-modules. A {\bf conformal linear map} from $U$ to $V$ is a $\mathbb{C}$-linear map $a: U\rightarrow V[\lambda]$, denoted by $a_\lambda: U\rightarrow V$, such that $[\partial, a_\lambda]=-\lambda a_\lambda$. Denote the vector space of all such maps by $Chom(U,V)$. It has a canonical structure of a $\mathbb{C}[\partial]$-module
$$(\partial a)_\lambda =-\lambda a_\lambda.$$ Define the {\bf conformal dual} of a $\mathbb{C}[\partial]$-module $U$ as
$U^{\ast c}=Chom(U,\mathbb{C})$ where $\mathbb{C}$ is viewed as
the trivial $\mathbb{C}[\partial]$-module, that is,
$$U^{\ast c}=\{a:U\rightarrow \mathbb{C}[\lambda]|\mathbb{C}\mbox{-linear~~and~~}a_\lambda(\partial b)=\lambda a_\lambda
b,\forall\ b\in U\}.$$}
\end{defi}

In the special case $U=V$, set $Cend(V)=Chom(V,V)$. If $V$ is a
finite $\mathbb{C}[\partial]$-module, then the
$\mathbb{C}[\partial]$-module $Cend(V)$ has a canonical structure
of an associative conformal algebra defined  by
\begin{eqnarray*}
(a_\lambda b)_\mu v=a_\lambda (b_{\mu-\lambda} v),\quad \forall\
a,b\in Cend(V), v\in V.\end{eqnarray*}

Set $a\rightharpoonup_\lambda v=l_A(a)_\lambda v$ and
$v\leftharpoonup_\lambda a=r_A(a)_{-\lambda-\partial} v$. Then a
structure of a bimodule $M$ over an associative conformal algebra
$A$ is the same as two $\mathbb{C}[\partial]$-module homomorphisms
$l_A$ and $r_A$ from $A$ to $Cend(M)$ such that the following
conditions hold:
\begin{eqnarray}
l_A(a_\lambda b)_{\lambda+\mu}v=l_A(a)_\lambda (l_A(b)_\mu v),\\
r_A(b)_{-\lambda-\mu-\partial}(r_A(a)_{-\lambda-\partial} v)=r_A(a_\mu b)_{-\lambda-\partial}v,\\
\label{eq1} l_A(a)_\lambda
(r_A(b)_{-\mu-\partial}v)=r_A(b)_{-\lambda-\mu-\partial}(l_A(a)_\lambda
v),\end{eqnarray} for all $a$, $b\in A$ and $v\in M$. We denote
this bimodule by $(M,l_A, r_A)$.

\begin{pro}\label{pr1}
Let $(M, l_A, r_A)$ be a finite bimodule of an associative conformal
algebra $A$. Let $l_A^\ast$ and $r_A^\ast$ be two
$\mathbb{C}[\partial]$-module homomorphisms from $A$ to
$Cend(M^{\ast c})$ defined by
\begin{eqnarray}
(l_A^\ast(a)_\lambda f)_\mu u=f_{\mu-\lambda}(l_A(a)_\lambda
u),~(r_A^\ast(a)_\lambda f)_\mu u=f_{\mu-\lambda}(r_A(a)_\lambda
u), \;\;\forall\ a\in A, f\in M^{\ast c}, u\in M.\end{eqnarray}
Then $(M^{\ast c},r_A^\ast, l_A^\ast)$ is a bimodule of $A$.
\end{pro}
\begin{proof}
Let $a$, $b\in A$, $f\in M^{\ast c}$ and $u\in M$. Since
\begin{eqnarray*}
(r_A^\ast(a)_\lambda(r_A^\ast(b)_\mu f))_\nu u&=&(r_A^\ast(b)_\mu
f)_{\nu-\lambda}(r_A(a)_\lambda u)
=f_{\nu-\lambda-\mu}(r_A(b)_\mu(r_A(a)_\lambda
u))\\&=&f_{\nu-\lambda-\mu}(r_A(a_\lambda b)_{\lambda+\mu}u)
=(r_A^\ast(a_\lambda b)_{\lambda+\mu} f)_\nu u;\\
(l_A^\ast(b)_{-\lambda-\mu-\partial}(l_A^\ast(a)_{-\lambda-\partial}
f))_\nu u
&=&(l_A^\ast(a)_{\mu}f)_{\lambda+\mu}(l_A(b)_{\nu-\lambda-\mu} u)
=f_\lambda(l_A(a)_\mu(l_A(b)_{\nu-\lambda-\mu}u))\\
&=&f_\lambda(l_A(a_\mu b)_{\nu-\lambda}u) =(l_A^\ast(a_\mu
b)_{\nu-\lambda} f)_\nu u=(l_A^\ast(a_\mu b)_{-\lambda-\partial}
f)_\nu u;\\
(r_A^\ast (a)_\lambda (l_A^\ast(b)_{-\mu-\partial} f))_\nu
u&=&(l_A^\ast (b)_{-\lambda-\mu+\nu}
f)_{\nu-\lambda}(r_A(a)_\lambda u)
=f_\mu(l_A(b)_{\nu-\lambda-\mu}(r_A(a)_\lambda u))\\&=&
f_\mu(r_A(a)_\lambda (l_A(b)_{\nu-\lambda-\mu}u))
=(r_A^\ast(a)_\lambda
f)_{\lambda+\mu}(l_A(b)_{\nu-\lambda-\mu}u)\\&=&(l_A^\ast(b)_{-\lambda-\mu-\partial}(r_A^\ast(a)_\lambda
f))_\nu u,
\end{eqnarray*}
we have
\begin{eqnarray*}
&&r_A^\ast(a)_\lambda(r_A^\ast(b)_\mu f)=r_A^\ast(a_\lambda
b)_{\lambda+\mu}
f,\;\;l_A^\ast(b)_{-\lambda-\mu-\partial}(l_A^\ast(a)_{-\lambda-\partial}
f) =l_A^\ast(a_\mu b)_{-\lambda-\partial} f,\\
&&r_A^\ast (a)_\lambda (l_A^\ast(b)_{-\mu-\partial}
f)=l_A^\ast(b)_{-\lambda-\mu-\partial}(r_A^\ast(a)_\lambda f).
\end{eqnarray*}
Hence $(M^{\ast c},r_A^\ast, l_A^\ast)$ is a bimodule of $A$.
\end{proof}

\begin{ex}
Let $A$ be a finite associative conformal algebra. Define two
$\mathbb{C}[\partial]$-module homomorphisms $L_A$ and $R_A$ from
$A$ to $Cend(A)$ by $L_A(a)_\lambda b=a_\lambda b$ and
$R_A(a)_\lambda b=b_{-\lambda-\partial}a$ for all $a$, $b\in A$.
Then $(A,L_A,R_A)$ is a bimodule of $A$. Moreover, $(A^{\ast
c},R_A^\ast,L_A^\ast)$ is a bimodule of $A$.
\end{ex}

\begin{pro}\label{pro1} {\rm (\cite[Proposition 4.4]{H1})}
Let $A$ and $B$ be two associative conformal algebras. Suppose that there are $\mathbb{C}[\partial]$-module homomorphisms $l_A$, $r_A: A\rightarrow Cend(B)$ and $l_B$, $r_B:B\rightarrow Cend(A)$ such that $(B, l_A, r_A)$ is a bimodule of $A$ and $(A, l_B, r_B)$ is a bimodule of $B$ and they satisfy the following relations:
\begin{eqnarray}
\label{es1}&&l_A(a)_\lambda(x_\mu y)=(l_A(a)_\lambda x)_{\lambda+\mu}y
+l_A(r_B(x)_{-\lambda-\partial}a)_{\lambda+\mu}y,\\
\label{es2}&&r_B(x)_{-\lambda-\mu-\partial}(a_\lambda b)
=a_\lambda (r_B(x)_{-\mu-\partial}b)+r_B(l_A(b)_\mu(x))_{-\lambda-\partial}a,\\
\label{es3}&&l_B(x)_\lambda(a_\mu b)=(l_B(x)_\lambda a)_{\lambda+\mu} b
+l_B(r_A(a)_{-\lambda-\partial}x)_{\lambda+\mu}b,\\
\label{es4}&&r_A(l_B(y)_\mu a)_{-\lambda-\partial}x+x_\lambda (r_A(a)_{-\mu-\partial} y)=r_A(a)_{-\lambda-\mu-\partial}(x_\lambda y),\\
\label{es5}&&r_A(r_B(y)_{-\mu-\partial}(a))_{-\lambda-\partial}x +x_\lambda (l_A(a)_\mu y)
=l_A(l_B(x)_\lambda a)_{\lambda+\mu} y+ (r_A(a)_{-\lambda-\partial}x)_{\lambda+\mu} y,\\
\label{es6}&&a_\lambda (l_B(x)_\mu b)+r_B(r_A(b)_{-\mu-\partial}x)_{-\lambda-\partial}(a)
=(r_B(x)_{-\lambda-\partial}a)_{\lambda+\mu}b+l_B(l_A(a)_\lambda x)_{\lambda+\mu}b,
\end{eqnarray}
for all $a$, $b\in A$ and $x$, $y\in B$. Then there is an associative conformal algebra structure on the $\mathbb{C}[\partial]$-module $A\oplus B$ given by
\begin{eqnarray}\label{ass-matched pair}
(a+x)_\lambda (b+y)=(a_\lambda b+l_B(x)_\lambda
b+r_B(y)_{-\lambda-\partial} a) +(x_\lambda y+l_A(a)_\lambda
y+r_A(b)_{-\lambda-\partial} x),\end{eqnarray} for all $a,b\in A$
and $x,y\in B$. We denote this associative conformal algebra by
$A\bowtie B$. $(A, B,l_A,r_A,l_B,r_B)$ satisfying the above
relations is called a {\bf matched pair } of associative conformal
algebras. Moreover, any associative conformal algebra $E=A\oplus
B$ where the sum is the direct sum of
$\mathbb{C}[\partial]$-modules and $A$, $B$ are two associative
conformal subalgebras of $E$, is $A\bowtie B$ associated to some
matched pair of associative conformal algebras.
\end{pro}

\begin{rmk}{
\rm If $l_B$, $r_B$ and the $\lambda$-product on $B$ are trivial
in Proposition \ref{pro1}, that is, $B$ is exactly a bimodule of
$A$, then $A\bowtie B$ is the {\bf semi-direct product} of $A$ and
its bimodule $B$, which is denoted by $A\ltimes_{l_A, r_A} B$.}
\end{rmk}

\section{Double constructions of Frobenius conformal algebras}
 We introduce the notion of double constructions of Frobenius conformal algebras. The relationship between double constructions of Frobenius conformal algebras and matched pairs of associative
 conformal algebras is investigated.

\begin{defi}{
\rm
Let $V$ be a $\mathbb{C}[\partial]$-module. A {\bf conformal bilinear form} on $V$ is a $\mathbb{C}$-bilinear map
$\langle \cdot, \cdot \rangle_\lambda: V\times V\rightarrow \mathbb{C}[\lambda]$ satisfying
\begin{eqnarray}
\langle \partial a, b\rangle_\lambda=-\lambda\langle a,
b\rangle_\lambda, ~~~\langle a, \partial
b\rangle_\lambda=\lambda\langle a, b\rangle_\lambda,\;\;\forall\
a,b\in V.
\end{eqnarray}
A conformal bilinear form is called {\bf symmetric} if $\langle
a,b\rangle_\lambda=\langle b,a\rangle_{-\lambda}$ for any $a$,
$b\in V$. }\end{defi}

Suppose that there is a conformal bilinear form on a
$\mathbb{C}[\partial]$-module $V$. Then we have a
$\mathbb{C}[\partial]$-module homomorphism $\varphi: V\longrightarrow
V^{\ast c},~~ v\mapsto \varphi_v$ defined by
$$(\varphi_v)_\lambda w=\langle v,w\rangle_\lambda,\quad \forall\ v,w\in V.$$
A conformal bilinear form is called {\bf non-degenerate} if $\varphi$
gives an isomorphism of $\mathbb{C}[\partial]$-modules between $V$
and $V^{\ast c}$.

\begin{defi}{
\rm An associative conformal algebra $A$ is called a {\bf
Frobenius conformal algebra} if there is a non-degenerate
conformal bilinear form on $A$ such that
\begin{eqnarray}
\langle a_\lambda b,c\rangle_\mu=\langle
a,b_{\mu-\partial}c\rangle_\lambda,\;\;\forall\ a,b,c\in A.
\end{eqnarray}
A Frobenius conformal algebra $A$ is called {\bf symmetric} if the
conformal bilinear form on $A$ is symmetric. }\end{defi}

\begin{ex}
Let $A=\mathbb{C}[\partial]a\oplus\mathbb{C}[\partial]b$. Suppose
that $A$ is an associative conformal algebra with the following
$\lambda$-product:
\begin{eqnarray*}
a_\lambda a=(\partial^2+\lambda\partial+\lambda^2)b,~~a_\lambda b=b_\lambda a=b_\lambda b=0.
\end{eqnarray*}
Then $A$ is a Frobenius conformal algebra with the following
conformal bilinear form:
\begin{eqnarray*}
\langle a, a\rangle_\lambda=\langle b, b\rangle_\lambda=0,~~~\langle a, b\rangle_\lambda=\langle b, a\rangle_\lambda=1.
\end{eqnarray*}
\end{ex}

\begin{ex}
Let $(A,\langle \cdot, \cdot\rangle)$ be a Frobenius algebra, that
is, $A$ is an associative algebra with a non-degenerate bilinear
form $\langle\cdot,\cdot\rangle$ satisfying
$$\langle ab, c\rangle=\langle a, bc\rangle,\;\;\forall\ a,b\in A.$$
Then $(\text{Cur}(A), \langle \cdot,\cdot \rangle_\lambda)$ is a Frobenius conformal algebra with $\langle \cdot, \cdot \rangle_\lambda$ defined by
\begin{eqnarray*}
\langle p(\partial)a, q(\partial)b\rangle_\lambda=p(-\lambda)q(\lambda)\langle a,b\rangle, ~~~\forall\ \text{$p(\partial)$, $q(\partial)\in \mathbb{C}[\partial]$, $a$, $b\in A$.}
\end{eqnarray*}
\end{ex}
\begin{defi}{\rm
If a Frobenius conformal algebra $A$ satisfies the following
conditions:
\begin{enumerate}
\item[(1)] $A=B\oplus B^{\ast c}$ where the sum is the direct sum
of $\mathbb{C}[\partial]$-modules; \item[(2)] $B$ and $B^{\ast c}$
are two associative conformal subalgebras of $A$;\item[(3)]  the
conformal bilinear form on $A$ is naturally given by
\begin{eqnarray}\label{dde1}
\langle a+f, b+g\rangle_\lambda =f_\lambda(b)+g_{-\lambda}(a),~~~~\forall\  a, b\in B, f, g\in B^{\ast c},
\end{eqnarray}
\end{enumerate}
then $A$ is called {\bf a double construction of Frobenius conformal algebra} associated to $B$ and $B^{\ast c}$.
}\end{defi}

\begin{thm}\label{thm1}
Let $A$ be a finite associative conformal algebra which is free as a $\mathbb{C}[\partial]$-module. Suppose that there is an associative conformal algebra structure on $A^{\ast c}$. Then there is a double construction of Frobenius conformal algebra associated to $A$ and $A^{\ast c}$ if and only if $(A, A^{\ast c}, R_A^\ast, L_A^\ast, R_{A^{\ast c}}^\ast, L_{A^{\ast c}}^\ast)$ is a matched pair of associative conformal algebras.
\end{thm}
\begin{proof}
Suppose that $(A, A^{\ast c}, R_A^\ast, L_A^\ast, R_{A^{\ast c}}^\ast, L_{A^{\ast c}}^\ast)$ is a matched pair of associative conformal algebras. Then $A\bowtie A^{\ast c}$ is endowed with an associative conformal algebra structure as follows.
\begin{eqnarray}
(a+f)_\lambda (b+g)=(a_\lambda b+R_{A^{\ast c}}^\ast(f)_\lambda
b+L_{A^{\ast c}}^\ast(g)_{-\lambda-\partial} a) +(f_\lambda
g+R_A^\ast(a)_\lambda g+L_A^\ast (b)_{-\lambda-\partial}f),
\end{eqnarray}
for all $a,b\in A$ and $f,g\in A^{\ast c}$. By this
$\lambda$-product, $A$ and $A^{\ast c}$ are two subalgebras of
$A\bowtie A^{\ast c}$.

Obviously, the conformal bilinear form on $A\bowtie A^{\ast c}$
given by (\ref{dde1}) is symmetric and non-degenerate.  For all
$a$, $b$, $c\in A$, and $f$, $g$, $h\in A^{\ast c}$, we have
\begin{eqnarray*}
&&\langle (a+f)_\lambda (b+g), c+h\rangle_\mu\\
&=&\langle a_\lambda b+R_{A^{\ast c}}^\ast (f)_\lambda b+ L_{A^{\ast c}}^\ast(g)_{-\lambda-\partial}a
+f_\lambda g+R_A^\ast(a)_\lambda g+L_A^\ast(b)_{-\lambda-\partial}f, c+h\rangle_\mu\\
&=&(f_\lambda g+R_A^\ast(a)_\lambda g+L_A^\ast(b)_{-\lambda-\partial}f)_\mu (c)
+h_{-\mu}(a_\lambda b+R_{A^{\ast c}}^\ast (f)_\lambda b+ L_{A^{\ast c}}^\ast(g)_{-\lambda-\partial}a)\\
&=&(f_\lambda g)_\mu c+g_{\mu-\lambda}(R_A(a)_\lambda c)
+f_\lambda(L_A(b)_{\mu-\lambda}c)+h_{-\mu}(a_\lambda b)\\
&&+(R_{A^{\ast c}}(f)_\lambda h)_{\lambda-\mu} b
+(L_{A^{\ast c}}(g)_{\mu-\lambda}h)_{-\lambda}a\\
&=&(f_\lambda g)_\mu c+g_{\mu-\lambda}(c_{-\lambda-\partial}a)
+f_\lambda(b_{\mu-\lambda} c)+h_{-\mu}(a_\lambda b)
+(h_{-\lambda-\partial}f)_{\lambda-\mu} b+(g_{\mu-\lambda} h)_{-\lambda}a\\
&=&(f_\lambda g)_\mu c+g_{\mu-\lambda}(c_{-\mu}a)
+f_\lambda(b_{\mu-\lambda} c)+h_{-\mu}(a_\lambda b)
+(h_{-\mu}f)_{\lambda-\mu} b+(g_{\mu-\lambda} h)_{-\lambda}a,
\end{eqnarray*}
and
\begin{eqnarray*}
&&\langle a+f, (b+g)_{\mu-\partial} (c+h)\rangle_\lambda\\
&=&\langle a+f, b_{\mu-\partial} c+R_{A^{\ast c}}^\ast(g)_{\mu-\partial}c
+L_{A^{\ast c}}^\ast(h)_{-\mu} b+g_{\mu-\partial}h
+R_A^\ast(b)_{\mu-\partial}h+L_A^\ast(c)_{-\mu}g \rangle_\lambda\\
&=& f_\lambda(b_{\mu-\partial} c)+f_\lambda(R_{A^{\ast c}}^\ast(g)_{\mu-\partial}c)
+f_\lambda(L_{A^{\ast c}}^\ast(h)_{-\mu} b)\\
&&+(g_{\mu-\partial}h)_{-\lambda}a+(R_A^\ast(b)_{\mu-\partial}h)_{-\lambda}a
+(L_A^\ast(c)_{-\mu}g)_{-\lambda}a\\
&=& f_\lambda(b_{\mu-\lambda} c)+(R_{A^{\ast c}}(g)_{\mu-\lambda} f)_\mu c
+(L_{A^{\ast c}}(h)_{-\mu}f)_{\lambda-\mu} b\\
&&+(g_{\mu-\lambda}h)_{-\lambda}a+h_{-\mu}(R_A(b)_{\mu-\lambda}a)
+g_{\mu-\lambda}(L_A(c)_{-\mu}a)\\
&&=f_\lambda(b_{\mu-\lambda} c)+(f_\lambda g)_\mu c
+(h_{-\mu}f)_{\lambda-\mu} b+(g_{\mu-\lambda}h)_{-\lambda}a+h_{-\mu}(a_\lambda b)
+g_{\mu-\lambda}(c_{-\mu}a).
\end{eqnarray*}
Hence this conformal bilinear form on $A\bowtie A^{\ast c}$ is
invariant. Therefore $A\bowtie A^{\ast c}$ is a double
construction of Frobenius conformal algebra associated to $A$ and
$A^{\ast c}$.

Conversely, suppose that there is a double construction of
Frobenius conformal algebra associated to $A$ and $A^{\ast c}$.
Therefore there is an associative conformal algebra structure on
$A\oplus A^{\ast c}$ associated to a matched pair $(A, A^{\ast
c},l_A,r_A,l_{A^{\ast c}},r_{A^{\ast c}})$. Note that in $A\oplus
A^{\ast c}$,
\begin{eqnarray*}
a_\lambda f=r_{A^{\ast c}}(f)_{-\lambda-\partial}a+l_A(a)_\lambda
f,\;\; f_\lambda a=l_{A^{\ast c}}(f)_\lambda
a+r_A(a)_{-\lambda-\partial} f,\;\;\forall\ a\in A, f\in A^{\ast
c}.
\end{eqnarray*}
For all $a$, $b\in A$, and $f\in A^{\ast c}$, we have
\begin{eqnarray*}
\langle {l_A(a)}_\lambda f, b\rangle_\mu&=&\langle a_\lambda f,
b\rangle_\mu=\langle a, f_{\mu-\partial}b\rangle_\lambda
=\langle f_{\mu-\lambda} b,a\rangle_{-\lambda}=\langle f, b_{-\mu}a\rangle_{\mu-\lambda}\\
&=& f_{\mu-\lambda}( {R_A(a)}_{\mu-\partial}
b)=(R_A^\ast(a)_\lambda f)_\mu b =\langle R_A^{\ast}(a)_\lambda f,
b\rangle_\mu.
\end{eqnarray*}
Since $\langle \cdot, \cdot\rangle_\mu$ is non-degenerate,
$l_A(a)_\lambda f=R_A^{\ast}(a)_\lambda f$ for all $a\in A$ and
$f\in A^{\ast c}$ and hence $l_A=R_A^{\ast}$. Similarly, we have
$$r_A=L_A^\ast, \;\;
l_{A^{\ast c}}=R_{A^{\ast c}}^\ast,\;\; r_{A^{\ast c}}=L_{A^{\ast
c}}^\ast.$$ Thus $(A, A^{\ast c}, R_A^\ast, L_A^\ast, R_{A^{\ast
c}}^\ast, L_{A^{\ast c}}^\ast)$ is a matched pair of associative
conformal algebras.
\end{proof}

\begin{thm}\label{t1}
Let $A$ be a finite associative conformal algebra which is free as
a $\mathbb{C}[\partial]$-module. Assume that there is an
associative conformal algebra structure on the
$\mathbb{C}[\partial]$-module $A^{\ast c}$. Then $(A, A^{\ast c},
R_A^\ast, L_A^\ast, R_{A^{\ast c}}^\ast, L_{A^{\ast c}}^\ast)$ is
a matched pair of associative conformal algebras if and only if
\begin{eqnarray}
\label{es7}R_A^\ast(a)_\lambda (f_\mu g)=(R_A^\ast(a)_\lambda f)_{\lambda+\mu}g
+R_A^\ast(L_{A^{\ast c}}^\ast(f)_{-\lambda-\partial} a)_{\lambda+\mu} g,\\
\label{es8}R_{A}^\ast(R_{A^{\ast c}}^\ast(f)_\lambda
a)_{\lambda+\mu} g +(L_A^\ast(a)_{-\lambda-\partial}
f)_{\lambda+\mu} g=L_{A}^\ast (L_{A^{\ast
c}}^\ast(g)_{-\mu-\partial}(a))_{-\lambda-\partial} f+f_\lambda
(R_A^\ast(a)_\mu g),
\end{eqnarray}
for all $a\in A$ and $f$, $g\in A^{\ast c}$.
\end{thm}
\begin{proof}
Obviously, (\ref{es7}) is exactly (\ref{es1}) and (\ref{es8}) is
exactly (\ref{es5}) when $$l_A=R_A^\ast, \;\;r_A=L_A^\ast,\;\;
l_B=R_{A^{\ast c}}^\ast,\;\;r_B=L_{A^{\ast c}}^\ast,$$ in
Proposition \ref{pro1}. Then the conclusion follows if we  prove
that (\ref{es7}), (\ref{es2}), (\ref{es3}) and (\ref{es4}) are
mutually equivalent, (\ref{es8}) and (\ref{es6}) are equivalent.
As an example, we give an explicit proof that (\ref{es8}) holds if
and only if (\ref{es6}) holds. The other cases can be proved
similarly.

Let $\{e_1,\cdots, e_n\}$ be a $\mathbb{C}[\partial]$-basis of $A$
and $\{e_1^\ast,\cdots,e_n^\ast\}$ be a dual
$\mathbb{C}[\partial]$-basis of $A^{\ast c}$ in the sense that
$(e_j^\ast)_\lambda e_i=\delta_{ij}$. Set
$${e_i}_\lambda
e_j=\sum_{k=1}^n P_k^{ij}(\lambda,\partial)e_k,\;\;{e_i^\ast}
_\lambda
e_j^\ast=\sum_{k=1}^nR_k^{ij}(\lambda,\partial)e_k^\ast,$$ where
$P_k^{ij}(\lambda,\partial)$ and $R_k^{ij}(\lambda,\partial)\in
\mathbb{C}[\lambda,\partial]$. Since
$$(L_A^\ast(e_i)_\lambda e_j^\ast)_\mu e_k={e_j^\ast}_{\mu-\lambda}({e_i}_\lambda e_k)={e_j^\ast}_{\mu-\lambda}(\sum_{j=1}^nP_j^{ik}(\lambda,\partial)e_j)=P_j^{ik}(\lambda,\mu-\lambda),$$
we have
\begin{eqnarray*}\label{w1}L_A^\ast(e_i)_\lambda e_j^\ast=\sum_{k=1}^nP_j^{ik}(\lambda,-\lambda-\partial)e_k^\ast.\end{eqnarray*}
Similarly, we have
\begin{eqnarray*}\label{w2}
&&L_{A^{\ast c}}^\ast(e_i^\ast)_\lambda e_j=\sum_{k=1}^n
R_j^{ik}(\lambda,-\lambda-\partial)e_k,\;\; R_A^\ast(e_i)_\lambda
e_j^\ast=\sum_{k=1}^nP_j^{ki}(\partial,-\lambda-\partial)e_k^\ast,\\
&&R_{A^{\ast c}}^\ast(e_i^\ast)_\lambda
e_j=\sum_{k=1}^nR_j^{ki}(\partial,-\lambda-\partial)e_k.\end{eqnarray*}
Therefore (\ref{es8}) holds for any $a\in A$, $f$, $g\in A^{\ast c}$ if and only if
\begin{eqnarray*}
&&(R_{A}^\ast(R_{A^{\ast c}}^\ast(e_j^\ast)_\lambda
e_i)_{\lambda+\mu} e_k^\ast
+(L_A^\ast(e_i)_{-\lambda-\partial} e_i^\ast)_{\lambda+\mu} e_k^\ast\\
&&-L_{A}^\ast (L_{A^{\ast
c}}^\ast(e_k^\ast)_{-\mu-\partial}e_i)_{-\lambda-\partial}
e_j^\ast+{e_j^\ast}_\lambda (R_A^\ast(e_i)_\mu e_k^\ast))_\nu
e_s=0,~~~~\forall~~i,j,k,s,
\end{eqnarray*}
if and only if the following equation holds:
\begin{eqnarray}
&&\sum_{t=1}^n(R_i^{tj}(-\lambda-\mu,\mu)P_k^{st}(-\nu,\nu-\lambda-\mu)+R_s^{tk}(\lambda+\mu,-\nu)P_j^{it}(\mu,\lambda)\nonumber\\
&&\label{w5}-R_i^{kt}(\nu-\lambda-\mu,\mu)P_j^{ts}(\nu-\lambda,\lambda)-R_s^{jt}(\lambda,-\nu)P_k^{ti}(\lambda-\nu,\nu-\lambda-\mu))=0,~~~~\forall~~i,j,k,s.
\end{eqnarray}
On the other hand, (\ref{es6}) holds for any $a$, $b\in A$, $x\in A^{\ast c}$ if and only if
\begin{eqnarray*}
{e_j^\ast}_\nu({e_i}_\lambda (R_{A^{\ast c}}^\ast(e_k^\ast)_\mu e_s)+L_{A^{\ast c}}^\ast(L_A^\ast(e_s)_{-\mu-\partial} e_k^\ast)_{-\lambda-\partial} e_i\\
-(L_{A^{\ast c}}^\ast(e_k^\ast)_{-\lambda-\partial} e_i)_{\lambda+\mu} e_s-R_{A^{\ast c}}^\ast(R_A^\ast(e_i)_\lambda e_k^\ast)_{\lambda+\mu} e_s)=0,~~~~\forall~~i,j,k,s,
\end{eqnarray*}
if and only if the following equation holds:
\begin{eqnarray}
&&\sum_{t=1}^n(R_i^{tj}(-\lambda-\nu,\lambda)P_k^{st}(-\lambda-\mu-\nu,\mu)+R_s^{tk}(\lambda+\nu,-\lambda-\mu-\nu)P_j^{it}(\lambda,\nu)\nonumber\\
&&\label{w6}-R_i^{kt}(\mu,\lambda)P_j^{ts}(\lambda+\mu,\nu)-R_s^{jt}(\nu,-\lambda-\mu-\nu)P_k^{ti}(-\lambda-\mu,\mu))=0,~~~~\forall~~i,j,k,s.
\end{eqnarray}
Note that (\ref{w6}) is exactly (\ref{w5}) when we replace
$\lambda$ by $\mu$, $\nu$ by $\lambda$, and $\mu$ by
$\nu-\lambda-\mu$ in (\ref{w6}). Therefore (\ref{es8}) holds if
and only if (\ref{es6}) holds.
\end{proof}

\section{Antisymmetric infinitesimal conformal bialgebras}
We introduce the notion of antisymmetric infinitesimal  conformal
bialgebras as (under certain conditions) the equivalent structures
of the matched pairs of associative conformal algebras as well as
double constructions of Frobenius conformal algebras given in the
previous section.

\begin{defi}{\rm
An {\bf associative conformal coalgebra} is a $\mathbb{C}[\partial]$-module $A$ endowed with a $\mathbb{C}[\partial]$-module homomorphism $\Delta: A\longrightarrow A\otimes A$ such that
\begin{equation}(I\otimes\Delta)\Delta(a)=(\Delta\otimes I)\Delta(a),\;\;\forall\ a\in A,\label{eq:coass}\end{equation}
where the module action of $\mathbb{C}[\partial]$ on $A\otimes A$ is defined as
$\partial (a\otimes b)=(\partial a)\otimes b+a\otimes (\partial b)$ for any $a$, $b\in A$.}\end{defi}

\begin{pro}\label{Dpro1}
Let $(A,\Delta)$ be a finite associative conformal coalgebra. Then
$A^{\ast c}=Chom(A,\mathbb{C})$ is endowed with an associative
conformal algebra structure with the following $\lambda$-product:
\begin{eqnarray}\label{eq3}(f_\lambda g)_\mu (a)=\sum f_\lambda(a_{(1)})g_{\mu-\lambda}(a_{(2)})=(f\otimes g)_{\lambda,\mu-\lambda}(\Delta(a)),\;\;\forall\  f, g\in A^{\ast c},\end{eqnarray}
where $\Delta(a)=\sum a_{(1)}\otimes a_{(2)}$ for all $a\in A$.
\end{pro}
\begin{proof}
By (\ref{eq3}), the conformal sesquilinearity of the
$\lambda$-product is naturally satisfied.  For all $a\in A$, $f$,
$g$ and $h\in A^{\ast c}$, we have
\begin{eqnarray*}
&&((f_\lambda g)_{\lambda+\mu}h-f_\lambda(g_\mu h))_\nu(a)\\
&&=\sum(f_\lambda g)_{\lambda+\mu}(a_{(1)})h_{\nu-\lambda-\mu}(a_{(2)})-\sum f_\lambda(a_{(1)})(g_\mu h)_{\nu-\lambda}(a_{(2)})\\
&&=\sum f_\lambda(a_{(11)})g_\mu(a_{(12)})h_{\nu-\lambda-\mu}(a_{(2)})-\sum f_\lambda(a_{(1)})g_\mu(a_{(21)})h_{\nu-\lambda-\mu}(a_{(22)})\\
&&=(f\otimes g\otimes h)_{\lambda, \mu,
\nu-\lambda-\mu}(((\Delta\otimes I)\Delta-(I\otimes
\Delta)\Delta)(a))=0.\end{eqnarray*} Hence the conclusion holds.
\end{proof}

\begin{pro}\label{pp1}
Let $A$ be a finite associative conformal algebra which is free as
a $\mathbb{C}[\partial]$-module, that is,
$A=\sum_{i=1}^n\mathbb{C}[\partial]e_i$, where
$\{e_1,\cdots,e_n\}$ is a $\mathbb{C}[\partial]$-basis of $A$.
Then $A^{\ast
c}=Chom(A,\mathbb{C})=\sum_{i=1}^n\mathbb{C}[\partial]e_i^\ast$ is
an associative conformal coalgebra with the following coproduct:
\begin{eqnarray}
\label{eq4}\Delta(f)=\sum_{i,j}f_\mu({e_i}_\lambda e_j)(e_i^\ast\otimes e_j^\ast)|_{\lambda=\partial\otimes 1,\mu=-\partial\otimes 1-1\otimes \partial},
\end{eqnarray}  where $\{e_1^\ast,\cdots, e_n^\ast\}$ is a dual $\mathbb{C}[\partial]$-basis of $A^{\ast c}$.  More precisely, if ${e_i}_\lambda e_j=\sum_kP_k^{ij}(\lambda,\partial)e^k$, then
$$\Delta(e_k^\ast)=\sum_{i,j}Q_k^{ij}(\partial\otimes 1, 1\otimes\partial)(e_i^\ast\otimes e_j^\ast),$$
where $Q_k^{ij}(x,y)=P_k^{ij}(x,-x-y)$.
\end{pro}
\begin{proof}
By (\ref{eq4}), $\Delta$ is a $\mathbb{C}[\partial]$-module
homomorphism. By the definition of $\Delta$, we have
\begin{eqnarray*}
&&(I\otimes\Delta)\Delta(e_k^\ast)-(\Delta\otimes I)\Delta(e_k^\ast)\\
&&=\sum_{i,j,l,r}Q_k^{ij}(\partial\otimes 1\otimes 1, 1\otimes\partial\otimes 1+1\otimes 1\otimes \partial)\\
&&\times
Q_j^{lr}(1\otimes\partial\otimes  1, 1\otimes 1\otimes \partial)(e_i^\ast\otimes e_l^\ast\otimes e_r^\ast) -\sum_{i,j,l,r}Q_k^{ij}(\partial\otimes 1\otimes 1+1\otimes\partial\otimes 1, 1\otimes 1\otimes \partial)\\
&&\times Q_i^{lr}(\partial\otimes 1\otimes  1, 1\otimes
\partial\otimes 1)(e_l^\ast\otimes e_r^\ast\otimes
e_j^\ast).\end{eqnarray*} On the other hand, since  ${e_i}_\lambda
({e_l}_\mu e_r)=({e_i}_\lambda e_l)_{\lambda+\mu} e_r$, we have
\begin{eqnarray*}\label{eq7}
\sum_j P_j^{lr}(\mu,\lambda+\partial)P_k^{ij}(\lambda,\partial)
=\sum_j
P_j^{il}(\lambda,-\lambda-\mu)P_k^{jr}(\lambda+\mu,\partial).
\end{eqnarray*}
Since $Q_k^{ij}(x,y)=P_k^{ij}(x,-x-y)$, we have
\begin{eqnarray}\label{eq8}
\sum_j
Q_j^{lr}(\mu,-\lambda-\mu-\partial)Q_k^{ij}(\lambda,-\lambda-\partial)
=\sum_j
Q_j^{il}(\lambda,\mu)Q_k^{jr}(\lambda+\mu,-\lambda-\mu-\partial).
\end{eqnarray}
Set $$\lambda=\partial \otimes 1\otimes 1, \;\;\mu=1\otimes
\partial \otimes 1,\;\;\partial=-\partial\otimes 1\otimes
1-1\otimes \partial\otimes 1-1\otimes 1\otimes \partial.$$ Then by
(\ref{eq8}), we have
\begin{eqnarray*}&&\label{eq6}\sum_jQ_k^{ij}(\partial\otimes 1\otimes 1, 1\otimes\partial\otimes 1+1\otimes 1\otimes \partial)\cdot Q_j^{lr}(1\otimes\partial\otimes  1, 1\otimes 1\otimes \partial)\\
&&=\sum_jQ_k^{jr}(\partial\otimes 1\otimes
1+1\otimes\partial\otimes 1, 1\otimes 1\otimes \partial)\cdot
Q_j^{il}(\partial\otimes 1\otimes  1, 1\otimes \partial\otimes
1).\nonumber\end{eqnarray*} Therefore for all $k\in
\{1,\cdots,n\}$, we have
\begin{eqnarray*}\label{eq5}(I\otimes\Delta)\Delta(e_k^\ast)=(\Delta\otimes I)\Delta(e_k^\ast).\end{eqnarray*}
Hence the conclusion holds.
\end{proof}

\begin{cor}
Let $(A=\mathbb{C}[\partial]a, \Delta)$ be an associative conformal coalgebra which is free of rank 1 as a $\mathbb{C}[\partial]$-module.
Then $\Delta(a)=k a\otimes a$ for some $k\in \mathbb{C}$.
\end{cor}
\begin{proof}
If $A=\mathbb{C}[\partial]a$ is an associative conformal algebra,
then $a_\lambda a=k a$ for some $k\in \mathbb{C}$. Therefore this
conclusion follows from Proposition \ref{pp1}.
\end{proof}

In the sequel, denote $\partial\otimes 1+1\otimes \partial$ by
$\partial^{\otimes^2}$ and $\partial\otimes 1\otimes 1+1\otimes
\partial\otimes 1+1\otimes 1\otimes \partial$ by
$\partial^{\otimes^3}$. Moreover, for any vector space $V$, let
$\tau:V\otimes V\rightarrow V\otimes V$ be the flip map, that is,
$$\tau(x\otimes y)=y\otimes x,\;\;\forall\ x,y\in V.$$

\begin{thm}\label{thm2}
Let $A$ be a finite associative conformal algebra which is free as a $\mathbb{C}[\partial]$-module. Suppose there is another associative conformal algebra structure on the $\mathbb{C}[\partial]$-module $A^{\ast c}$ obtained from a $\mathbb{C}[\partial]$-module homomorphism $\Delta:A\rightarrow A\otimes A$. Then $(A,A^{\ast c}, R_A^\ast, L_A^\ast, R_{A^{\ast c}}^\ast, L_{A^{\ast c}}^\ast)$ is a matched pair of associative conformal algebras if and only if $\Delta$ satisfies
\begin{eqnarray}
\label{thq1}\Delta(a_\lambda b)=(I\otimes {L_A(a)}_\lambda)\Delta(b)+({R_A(b)}_{-\lambda-\partial^{\otimes^2}}\otimes I)\Delta(a),\\
\label{thq2}({L_A(b)}_{-\lambda-\partial^{\otimes^2}}\otimes I-I\otimes R_A(b)_{-\lambda-\partial^{\otimes^2}})\Delta(a)+\tau({L_A(a)}_\lambda \otimes I-I\otimes {R_A(a)}_\lambda)\Delta(b)=0,
\end{eqnarray}
for all $a$, $b\in A$.   \end{thm}
\begin{proof}
With the same assumption as that in the proof of Theorem \ref{t1} and by Proposition \ref{pp1}, we have
\begin{eqnarray*}\Delta(e_k)=\sum_{i,j}Q_k^{ij}(\partial\otimes 1,1\otimes \partial) e_i\otimes e_j,~~{\rm where}~~Q_k^{ij}(x,y)=R_k^{ij}(x,-x-y).\end{eqnarray*}
Considering the coefficient of $e_j\otimes e_k$ in
\begin{eqnarray*}
\Delta({e_s}_\lambda e_i)=(I\otimes {L_A(e_s)}_\lambda)\Delta(e_i)+({R_A(e_i)}_{-\lambda-\partial\otimes 1-1\otimes \partial}\otimes I)\Delta(e_s),
\end{eqnarray*}
we have
\begin{eqnarray}
&&\sum_{i=1}^nP_t^{si}(\lambda,\partial\otimes 1+1\otimes \partial)R_t^{jk}(\partial\otimes 1,-1\otimes \partial-\partial\otimes 1)\nonumber\\
&&=\sum_{i=1}^nR_i^{jt}(\partial\otimes 1,-\lambda-\partial\otimes 1-1\otimes \partial)P_k^{st}(\lambda,1\otimes \partial)\nonumber\\
&&+\sum_{i=1}^nR_s^{tk}(-\lambda-1\otimes
\partial,\lambda)P_j^{ti}(\lambda+1\otimes
\partial,\partial\otimes 1).\label{pq1}
\end{eqnarray}
On the other hand, (\ref{es7}) holds for any $a\in A$, $f$, $g\in A^{\ast c}$ if and only if
\begin{eqnarray*}(R_A^\ast(e_i)_\lambda ({e_j^\ast}_\mu e_k^\ast))_\nu e_s=((R_A(e_i)_\lambda e_j^\ast)_{\lambda+\mu}e_k^\ast
+R_A(L_{A^{\ast c}}^\ast(e_j^\ast)_{-\lambda-\partial} e_i)_{\lambda+\mu} e_k^\ast)_\nu e_s,~~~~\forall~~i,j,k,s,
\end{eqnarray*}
if and only if the following equation holds:
\begin{eqnarray}
\sum_{i=1}^nP_t^{si}(-\nu,\nu-\lambda)R_t^{jk}(\mu,\lambda-\nu)&=&\sum_{i=1}^nR_i^{jt}(\mu,\lambda)P_k^{st}(-\nu, \nu-\lambda-\mu)\nonumber\\
\label{pq2}&&+\sum_{i=1}^nR_s^{tk}(\lambda+\mu,-\nu)P_j^{ti}(-\lambda-\mu,\mu)),~~~~~\forall~~j,k,s,t.
\end{eqnarray}
Obviously, (\ref{pq1}) is exactly (\ref{pq2}) by replacing
$\lambda$ by $-\nu$, $1\otimes \partial$ by $\nu-\lambda-\mu$, and
$\partial\otimes 1$ by $\mu$ in (\ref{pq1}). Then (\ref{thq1})
holds if and only if (\ref{es7}) holds. With a similar discussion,
(\ref{thq2}) holds if and only if (\ref{es8}) holds. Then this
conclusion follows from Theorem \ref{t1}.
\end{proof}

\begin{defi}\label{de1}{\rm
Let $(A,\cdot_\lambda \cdot)$ be an associative conformal algebra and $(A,\Delta)$ be an associative conformal coalgebra. If $\Delta$ satisfies (\ref{thq1}), then $(A, \cdot_\lambda\cdot, \Delta)$ is called an {\bf infinitesimal conformal bialgebra}. Moreover, if $\Delta$ also satisfies (\ref{thq2}), then $(A, \cdot_\lambda\cdot, \Delta)$ is called an {\bf antisymmetric infinitesimal (ASI) conformal bialgebra}.
}\end{defi}
\begin{rmk}{\rm
By the definition of infinitesimal bialgebra given in \cite{A1},
the corresponding definition of conformal version of infinitesimal
bialgebra should be given as follows. $(A,\circ_\lambda )$ is an
associative conformal algebra and $(A,\Delta)$ is an associative
conformal coalgebra satisfying
\begin{eqnarray}
\Delta(a\circ_\lambda b)=({L_A(a)}_\lambda \otimes
I)\Delta(b)+(I\otimes
{R_A(b)}_{-\lambda-\partial^{\otimes^2}})\Delta(a),\;\;\forall\
a,b\in A.
\end{eqnarray}
We would like to point out that the definition of infinitesimal
conformal bialgebra given in Definition \ref{de1} is equivalent to
this definition if we replace $(A,\circ_\lambda )$ by the opposite
algebra of $(A,\cdot_\lambda \cdot)$ in the sense that
$a\circ_\lambda b=b_{-\lambda-\partial}a$ for all $a$, $b\in A$.
}\end{rmk}

Combining Theorems~\ref{thm1} and ~\ref{thm2} together, we have the following conclusion.

\begin{cor}\label{cc1}
Let $A$ be a finite associative conformal algebra which is free as a $\mathbb{C}[\partial]$-module. Suppose there is another associative conformal algebra structure on the $\mathbb{C}[\partial]$-module $A^{\ast c}$ obtained from a $\mathbb{C}[\partial]$-module homomorphism $\Delta:A\rightarrow A\otimes A$. Then the following conditions are equivalent:
\begin{enumerate}
\item[(1)] $(A,\cdot_\lambda\cdot, \Delta)$ is an ASI conformal
bialgebra; \item[(2)] There is a double construction of Frobenius
conformal algebra associated to $A$ and $A^{\ast c}$;\item[(3)]
$(A, A^{\ast c}, R_A^\ast, L_A^\ast, R_{A^{\ast c}}^\ast,
L_{A^{\ast c}}^\ast)$ is a matched pair of associative conformal
algebras.
\end{enumerate}
\end{cor}

\begin{pro}\label{cc2}
Let $(A,\cdot_\lambda \cdot, \Delta_A)$ be a finite ASI conformal bialgebra, where $A$ is free as a $\mathbb{C}[\partial]$-module. Then
$(A^{\ast c}, \cdot_\lambda \cdot, \Delta_{A^{\ast c}})$ is also an ASI conformal bialgebra, where $\cdot_\lambda \cdot$ and $\Delta_{A^{\ast c}}$ are on $A^{\ast c}$ given by (\ref{eq3}) and (\ref{eq4}) respectively.
\end{pro}
\begin{proof}
It follows directly from Corollary \ref{cc1} and the symmetry of
the associative conformal algebra $A$ and $A^{\ast c}$ in the
double construction $A\bowtie A^{\ast c}$ of Frobenius conformal
algebra associated to $A$ and $A^{\ast c}$.
\end{proof}

At the end of this section, we present two examples of ASI
conformal bialgebras.

\begin{ex}
Let $(A,\cdot, \overline{\Delta})$ be an antisymmetric infinitesimal bialgebra given in \cite{Bai1}. Then $\text{Cur}(A)$ has a natural ASI conformal bialgebra structure defined by
\begin{eqnarray*}
\Delta(p(\partial)a)=p(\partial\otimes 1+1\otimes \partial)\overline{\Delta}(a),\;\;\forall\ \text{$p(\partial)\in \mathbb{C}[\partial]$, $a\in A$.}
\end{eqnarray*}
\end{ex}

\begin{ex}
Let $p(\lambda)\in \mathbb{C}[\lambda]$ and
$A=\mathbb{C}[\partial]a\oplus \mathbb{C}[\partial]b$ be a rank
two associative conformal algebra with the following
$\lambda$-product:
\begin{eqnarray*}
a_\lambda a=p(\lambda+\partial)b,~~~a_\lambda b=b_\lambda a=b_\lambda b=0.
\end{eqnarray*}
Define a $\mathbb{C}[\partial]$-module homomorphism
$\Delta:A\rightarrow A\otimes A$ by
\begin{eqnarray*}
\Delta(a)=a\otimes b,~~~\Delta(b)=b\otimes b.
\end{eqnarray*}
By a straightforward computation, $(A,\cdot_\lambda \cdot,\Delta)$
is an ASI conformal bialgebra if and only if $p(\lambda)$ is an
odd polynomial, i.e., $p(\lambda)=-p(-\lambda)$.
\end{ex}

\section{Coboundary antisymmetric infinitesimal conformal bialgebras}
We consider a special class of ASI conformal bialgebras which are called coboundary ASI conformal bialgebras.

\begin{defi}{\rm
For an ASI conformal bialgebra $(A,\cdot_\lambda \cdot, \Delta)$, if there exists $r\in A\otimes A$ such that
\begin{eqnarray}\label{eQ1}
\Delta(a)=(I\otimes L_A(a)_\lambda-R_A(a)_{\lambda}\otimes
I)r|_{\lambda=-\partial^{\otimes^2}}, \;\;\forall\ a\in A,
\end{eqnarray}
then $(A,\cdot_\lambda \cdot, \Delta)$ is called {\bf coboundary}.}
\end{defi}

For $r=\sum_i r_i\otimes l_i\in A\otimes A$, define
\begin{eqnarray}
r\bullet r=\sum_{i,j} r_i\otimes r_j\otimes {l_i}_\mu l_j|_{\mu=\partial\otimes 1\otimes 1}-r_i\otimes {r_j}_\mu l_i\otimes l_j|_{\mu=-\partial^{\otimes^2}\otimes 1}+{r_i}_\mu r_j\otimes l_i\otimes l_j|_{\mu=1\otimes \partial\otimes 1}.
\end{eqnarray}

\begin{pro}\label{ppr1}
Let $A$ be an associative conformal algebra and $r=\sum_i r_i\otimes l_i\in A\otimes A$. Then the map defined by (\ref{eQ1}) gives an associative conformal coalgebra $(A,\Delta)$ if and only if
\begin{eqnarray}\label{qw1}
(I\otimes I\otimes
L_A(a)_{-\partial^{\otimes^3}}-R_A(a)_{-\partial^{\otimes^3}}\otimes
I\otimes I) (r\bullet r)=0,\;\;\forall\ a\in A.
\end{eqnarray}
\end{pro}

\begin{proof}
By the definition of $\Delta$, we have {\small \begin{eqnarray*}
(I\otimes\Delta)\Delta(a)&=&(I\otimes\Delta)(\sum_i r_i\otimes a_\lambda l_i-\sum_i{r_i}_{-\lambda-\partial\otimes 1}a\otimes l_i)|_{\lambda={-\partial^{\otimes^2}}}\\
&=&(I\otimes\Delta)(\sum_i r_i\otimes a_{-\partial^{\otimes^2}}l_i-\sum_i{r_i}_{1\otimes \partial} a\otimes l_i)\\
&=&\sum_ir_i\otimes \Delta(a_{-\partial^{\otimes^2}}l_i)
-\sum_i{r_i}_{1\otimes \partial\otimes 1+1\otimes 1\otimes \partial} a\otimes \Delta(l_i)\\
&=&\sum_{i,j}(r_i\otimes r_j\otimes (a_{-\partial^{\otimes^2}\otimes 1}l_i)_\mu l_j-r_i\otimes {r_j}_{-\mu-1\otimes \partial\otimes 1}(a_{-\partial^{\otimes^2}\otimes 1}l_i)\otimes l_j\\
&&-({r_i}_{1\otimes \partial\otimes 1+1\otimes 1\otimes \partial} a\otimes r_j\otimes {l_i}_\mu l_j-{r_i}_{1\otimes \partial\otimes 1+1\otimes 1\otimes \partial} a\otimes {r_j}_{-\mu-1\otimes \partial\otimes 1} {l_i}\otimes l_j))_{\mu={-1\otimes\partial^{\otimes^2}}}\\
&=&\sum_{i,j}r_i\otimes r_j\otimes a_{-\partial^{\otimes^3}}({l_i}_{\partial\otimes 1\otimes 1} l_j)
-\sum_{i,j}r_i\otimes {r_j}_{1\otimes 1\otimes \partial}(a_{-\partial^{\otimes^2}\otimes 1}l_i)\otimes l_j\\
&&-\sum_{i,j} {r_i}_{1\otimes \partial^{\otimes^2}} a\otimes r_j\otimes {l_i}_{-1\otimes \partial^{\otimes^2}}l_j+\sum_{i,j} {r_i}_{1\otimes \partial\otimes 1+1\otimes 1\otimes \partial} a \otimes {r_j}_{1\otimes 1\otimes \partial}l_i\otimes l_j\\
&=&(1\otimes 1\otimes L_A(a)_{-\partial^{\otimes^3}})(\sum_{i,j}r_i\otimes r_j\otimes {l_i}_{\partial\otimes 1\otimes 1}l_j)-\sum_{i,j}r_i\otimes {r_j}_{1\otimes 1\otimes \partial}(a_{-\partial^{\otimes^2}\otimes 1}l_i)\otimes l_j\\
&&-(R_A(a)_{-\partial^{\otimes^3}}\otimes 1\otimes 1)\sum_{i,j}(r_i\otimes r_j\otimes {l_i}_{\partial\otimes 1\otimes 1}l_j-r_i\otimes {r_j}_{1\otimes 1\otimes \partial} l_i\otimes l_j).
\end{eqnarray*}}
Similarly, we have  \begin{eqnarray*}
(\Delta\otimes I)\Delta(a)&=&(\Delta\otimes I)\sum_i (r_i\otimes a_{-\partial^{\otimes^2}} l_i-{r_i}_{1\otimes \partial} a\otimes l_i)\\
&=&\sum_i(\Delta(r_i)\otimes a_{-\partial^{\otimes^3}}l_i-\Delta({r_i}_{1\otimes \partial} a)\otimes l_i)\\
&=&\sum_{i,j}(r_j\otimes {r_i}_{-\partial^{\otimes^2}\otimes 1}l_j\otimes a_{-\partial^{\otimes^3}}l_i
-{r_j}_{1\otimes \partial\otimes 1}r_i\otimes l_j\otimes a_{-\partial^{\otimes^3}}l_i)\\
&&-\sum_{i,j}(r_j\otimes ({r_i}_{1\otimes 1\otimes \partial}a)_{-\partial^{\otimes^2}\otimes 1}l_j\otimes l_i-{r_j}_{1\otimes \partial\otimes 1}({r_i}_{1\otimes 1\otimes \partial}a)\otimes l_j\otimes l_i)\\
&=& (I\otimes I\otimes L_A(a)_{-\partial^{\otimes^3}})\sum_{i,j} (r_i\otimes {r_j}_{-\partial^{\otimes^2}\otimes 1}l_i\otimes l_j-{r_i}_{1\otimes \partial\otimes 1} r_j\otimes l_i\otimes l_j)\\
&&-\sum_{i,j}r_i\otimes {r_j}_{1\otimes 1\otimes \partial}(a_{-\partial^{\otimes^2}\otimes 1}l_i)\otimes l_j\\
&&+(R_A(a)_{-\partial^{\otimes^3}}\otimes 1\otimes 1)\sum_{i,j}
{r_i}_{1\otimes \partial \otimes 1} r_j\otimes l_i \otimes l_j.
\end{eqnarray*}
Therefore $\Delta$ satisfies (\ref{eq:coass}) if and only if
(\ref{qw1}) holds.
\end{proof}

\begin{thm}\label{tt2}
Let $(A,\cdot_\lambda\cdot)$ be an associative conformal algebra
and $r=\sum_i r_i\otimes l_i\in A\otimes A$. Then the map defined
by (\ref{eQ1}) gives an associative conformal coalgebra
$(A,\Delta)$ such that $(A, \cdot_\lambda \cdot, \Delta)$ is an
ASI conformal bialgebra if and only if (\ref{qw1}) and the
following equation hold:
\begin{eqnarray}\label{thq3}
(L_A(b)_{-\lambda-\partial^{\otimes^2}}\otimes I-I\otimes R_A(b)_{-\lambda-\partial^{\otimes^2}})
(I\otimes L_A(a)_{-\partial^{\otimes^2}}-R_A(a)_{-\partial^{\otimes^2}}\otimes I)(r+\tau r)=0,
\end{eqnarray}
for all $a$, $b\in A$.
\end{thm}
\begin{proof}
Let $a,b\in A$. We first check that (\ref{thq1})
holds automatically. In fact, we have
\begin{eqnarray*}
&&(I\otimes {L_A(a)}_\lambda)\Delta(b)+({R_A(b)}_{-\lambda-\partial^{\otimes^2}}\otimes I)\Delta(a)\\
&=& (I\otimes {L_A(a)}_\lambda)\sum_i(r_i\otimes b_{-\partial^{\otimes^2}}l_i-{r_i}_{1\otimes \partial} b\otimes l_i)\\
&&+(R_A(b)_{-\lambda-\partial^{\otimes^2}}\otimes I)\sum_i(r_i\otimes a_{-\partial^{\otimes^2}}l_i
-{r_i}_{1\otimes \partial} a\otimes l_i)\\
&=&\sum_i(r_i\otimes a_\lambda (b_{-\partial^{\otimes^2}}l_i)-{r_i}_{\lambda+1\otimes \partial} b\otimes a_\lambda l_i\\
&&+{r_i}_{\lambda+1\otimes \partial}b\otimes a_\lambda l_i-({r_i}_{1\otimes \partial} a)_{\lambda+1\otimes \partial} b\otimes l_i)\\
&=&\sum_i(r_i\otimes a_\lambda (b_{-\partial^{\otimes^2}}l_i)-({r_i}_{1\otimes \partial} a)_{\lambda+1\otimes \partial} b\otimes l_i)\\
&=&\sum_i(r_i\otimes (a_\lambda b)_{-\partial^{\otimes^2}}l_i-{r_i}_{1\otimes \partial} (a_\lambda b)\otimes l_i)\\
&=& \Delta(a_\lambda b).
\end{eqnarray*}
Therefore (\ref{thq1})
holds.

Obviously, we have
\begin{eqnarray*}
&&({L_A(b)}_{-\lambda-\partial^{\otimes^2}}\otimes I-I\otimes R_A(b)_{-\lambda-\partial^{\otimes^2}})\Delta(a)\\
&&=({L_A(b)}_{-\lambda-\partial^{\otimes^2}}\otimes I-I\otimes
R_A(b)_{-\lambda-\partial^{\otimes^2}}) (I\otimes
L_A(a)_{-\partial^{\otimes^2}}-R_A(a)_{-\partial^{\otimes^2}}\otimes
I)r.
\end{eqnarray*}
Moreover, we have
\begin{eqnarray*}
&&\tau({L_A(a)}_\lambda \otimes I-I\otimes {R_A(a)}_\lambda)\Delta(b)\\
&=&\tau ({L_A(a)}_\lambda \otimes I-I\otimes
{R_A(a)}_\lambda)\sum_i(r_i\otimes b_{-\partial^{\otimes^2}}l_i
-{r_i}_{1\otimes \partial}b \otimes l_i)\\
&=&\tau \sum_i(a_\lambda r_i\otimes
b_{-\lambda-\partial^{\otimes^2}}l_i-a_\lambda ({r_i}_{1\otimes
\partial}b)\otimes l_i-r_i\otimes
(b_{-\partial^{\otimes^2}}l_i)_{-\lambda-1\otimes \partial}a
+{r_i}_{\lambda+1\otimes \partial}b\otimes{l_i}_{-\lambda-1\otimes \partial}a\\
&=&\sum_i( b_{-\lambda-\partial^{\otimes^2}}l_i\otimes a_\lambda
r_i-l_i\otimes a_\lambda ({r_i}_{\partial\otimes
1}b)-(b_{-\partial^{\otimes^2}}l_i)_{-\lambda-\partial\otimes
1}a\otimes r_i
+{l_i}_{-\lambda-\partial\otimes 1}a\otimes {r_i}_{\lambda+\partial\otimes 1}b\\
&=&\sum_i( b_{-\lambda-\partial^{\otimes^2}}l_i\otimes a_\lambda
r_i-l_i\otimes (a_\lambda r_i) _{\lambda+\partial\otimes
1}b)-b_{-\lambda-\partial^{\otimes^2}}({l_i}_{1\otimes \partial}a)
+{l_i}_{-\lambda-\partial\otimes 1}a\otimes {r_i}_{\lambda+\partial\otimes 1}b\\
&=&({L_A(b)}_{-\lambda-\partial^{\otimes^2}}\otimes I-I\otimes
R_A(b)_{-\lambda-\partial^{\otimes^2}})(I\otimes
L_A(a)_{-\partial^{\otimes^2}}-R_A(a)_{-\partial^{\otimes^2}}\otimes
I)(\tau r).
\end{eqnarray*}
Therefore  (\ref{thq2}) holds if and only if (\ref{thq3}) holds.
Hence by Proposition \ref{ppr1}, the conclusion holds.
\end{proof}

\begin{cor}
Let $A$ be an associative conformal algebra and $r\in A\otimes A$. Suppose that $r$ is antisymmetric. Then
the map defined by (\ref{eQ1}) gives an associative conformal coalgebra $(A,\Delta)$ such that $(A, \cdot_\lambda \cdot, \Delta)$ is an ASI conformal bialgebra if
\begin{eqnarray}\label{qw2}
r\bullet r\equiv 0~~\text{mod}~~(\partial^{\otimes^3}).
\end{eqnarray}
\end{cor}
\begin{proof}
If $r\bullet r\equiv 0~~\text{mod}~~(\partial^{\otimes^3})$, (\ref{qw1}) naturally holds by conformal sesquilinearity. Then this conclusion directly follows from Theorem \ref{tt2}.
\end{proof}

\begin{defi}{\rm Let $A$ be an associative conformal algebra and $r\in A\otimes
A$. (\ref{qw2}) is called {\bf associative conformal Yang-Baxter
equation} in $A$.}
\end{defi}

\begin{rmk}{\rm In fact, the associative conformal Yang-Baxter
equation (\ref{qw2}) is regarded as a conformal analogue of the
associative Yang-Baxter equation in an associative algebra
(\cite{Bai1}).}
\end{rmk}

\begin{defi}
{\rm Let $(A,\cdot_\lambda\cdot, \Delta_A)$ and $(B,\cdot_\lambda\cdot, \Delta_{B})$ be two ASI conformal bialgebras. If $\varphi: A\rightarrow B$ is a homomorphism of associative conformal algebras satisfying
\begin{eqnarray}
(\varphi\otimes
\varphi)\Delta_A(a)=\Delta_B(\varphi(a)),~~\forall\ a\in A,
\end{eqnarray}
then $\varphi$ is called {\bf a homomorphism of ASI conformal bialgebras}.}
\end{defi}


\begin{thm}\label{double-1}
Let $(A,\cdot_\lambda\cdot, \Delta_A)$ be a finite ASI conformal
bialgebra where $A$ is free as a $\mathbb{C}[\partial]$-module.
Then there is a canonical ASI conformal bialgebra structure on the
$\mathbb{C}[\partial]$-module $A\oplus A^{\ast c}$ such that the
inclusions $i_1: A\rightarrow A\oplus A^{\ast c}$ and $i_2:
A^{\ast c}\rightarrow A\oplus A^{\ast c}$ are homomorphisms of ASI
conformal bialgebras. Here the ASI conformal bialgebra structure
on $A^{\ast c}$ is $(A^{\ast c}, \cdot_\lambda \cdot,
-\Delta_{A^{\ast c}})$, where $\cdot_\lambda\cdot$ and
$\Delta_{A^{\ast c}}$ on $A^{\ast c}$ are defined by (\ref{eq3})
and (\ref{eq4}) respectively.
\end{thm}
\begin{proof}
Let $\{e_1,\cdots,e_n\}$ be a $\mathbb{C}[\partial]$-basis of $A$ and $\{e_1^\ast,\cdots,e_n^\ast\}$ be the dual basis in $A^{\ast c}$. Since $(A,\cdot_\lambda\cdot, \Delta_A)$ is an ASI conformal bialgebra, $(A,A^{\ast c}, R_A^\ast, L_A^\ast, R_{A^{\ast c}}^\ast, L_{A^{\ast c}}^\ast)$ is a matched pair of associative conformal algebras by Corollary \ref{cc1}, where the associative conformal algebra structure on $A^{\ast c}$ is obtained from $\Delta_A$.
Then by Proposition \ref{pro1}, there is an associative conformal algebra structure on the $\mathbb{C}[\partial]$-module $A\oplus A^{\ast c}$ associated with such a matched pair.
Set
\begin{eqnarray}
{e_i}_\lambda e_j=\sum_kP_k^{ij}(\lambda,\partial)e_k,~~~~\Delta_A(e_j)=\sum_{i,k}R_j^{ik}(\partial\otimes 1,-\partial\otimes1-1\otimes \partial)e_i\otimes e_k,
\end{eqnarray}
where $P_k^{ij}(\lambda,\partial)$, $R_j^{ik}(\lambda,\partial)\in
\mathbb{C}[\lambda,\partial]$. Then by Proposition \ref{pro1}
again, we have
\begin{eqnarray*}
{e_i^\ast}_\lambda e_j^\ast
&=&\sum_kR_k^{ij}(\lambda,\partial)e_k^\ast,\\
{e_i}_\lambda e_j^\ast&=&R_A^\ast(e_i)_\lambda e_j^\ast+L_{A^{\ast
c}}^\ast(e_j^\ast)_{-\lambda-\partial}e_i =\sum_k
P_j^{ki}(\partial,-\lambda-\partial)e_k^\ast+\sum_k
R_i^{jk}(-\lambda-\partial,\lambda)e_k,\\
{e_i^\ast}_\lambda e_j&=&R_{A^{\ast c}}^\ast(e_i^\ast)_\lambda
e_j+L_{A^{\ast c}}^\ast(e_j)_{-\lambda-\partial}e_i^\ast =\sum_k
R_j^{ki}(\partial,-\lambda-\partial)e_k+\sum_k
P_i^{jk}(-\lambda-\partial,\lambda)e_k^\ast.
\end{eqnarray*}

Let $r=\sum_{i=1}^n e_i\otimes e_i^\ast\in (A\oplus A^{\ast
c})\otimes (A\oplus A^{\ast c})$ and define
$$\Delta_{A\oplus A^{\ast c}}(a)=(I\otimes L_A(a)_\lambda-R_A(a)_\lambda\otimes I)r|_{\lambda=-\partial^{\otimes^2}},\;\;\forall\ a\in A\oplus A^{\ast c}.$$
Note that
\begin{eqnarray*}
r\bullet r&=&\sum_{i,j}(e_i\otimes e_j\otimes {e_i^\ast}_\mu e_j^\ast|_{\mu=\partial\otimes 1\otimes 1}\\
&&-e_i\otimes {e_j}_\mu e_i^\ast\otimes
e_j^\ast|_{\mu=-\partial^{\otimes^2}\otimes 1}
+{e_i}_\mu e_j\otimes e_i^\ast\otimes e_j^\ast|_{\mu=1\otimes \partial\otimes 1})\\
&=&\sum_{i,j,k}(R_k^{ij}(\partial\otimes 1\otimes 1,1\otimes 1\otimes \partial)  e_i\otimes e_j\otimes e_k^\ast-P_i^{kj}(1\otimes \partial\otimes 1, \partial\otimes 1\otimes 1)e_i\otimes e_k^\ast\otimes e_j^\ast\\
&&-R_j^{ik}(\partial\otimes 1\otimes 1,-\partial^{\otimes^2}\otimes 1)e_i\otimes e_k\otimes e_j^\ast+P_k^{ij}(1\otimes \partial\otimes 1,\partial\otimes 1\otimes 1)e_k\otimes e_i^\ast\otimes e_j^\ast)\\
&=&\sum_{i,j,k}(R_k^{ij}(\partial\otimes 1\otimes 1,1\otimes 1\otimes \partial)-R_k^{ij}(\partial\otimes 1\otimes 1,-\partial^{\otimes^2}\otimes 1))e_i\otimes e_i\otimes e_k^\ast\\
&\equiv& 0 ~~~\text{mod}~~~\partial^{\otimes^3}.
\end{eqnarray*}
Then (\ref{qw1}) holds by conformal sesquilinearity. Moreover, for
all $e_i\in A$, we have
\begin{eqnarray*}
&&(I\otimes L_A(e_i)_{-\partial^{\otimes^2}}-R_A(e_i)_{-\partial^{\otimes^2}}\otimes I)(r+\tau r)\\
&=&\sum_j(I\otimes L_A(e_i)_{-\partial^{\otimes^2}}-R_A(e_i)_{-\partial^{\otimes^2}}\otimes I)(e_j\otimes e_j^\ast+e_j^\ast\otimes e_j)\\
&=&\sum_j(e_j\otimes
{e_i}_{-\partial^{\otimes^2}}e_j^\ast+e_j^\ast\otimes{e_i}_{-\partial^{\otimes^2}}
e_j
-{e_j}_{1\otimes \partial}e_i\otimes e_j^\ast-{e_j^\ast}_{1\otimes \partial}e_i\otimes e_j)\\
&=&\sum_{j,k}(P_j^{ki}(1\otimes \partial,\partial\otimes
1)e_j\otimes e_k^\ast
+R_i^{jk}(\partial\otimes 1,-\partial^{\otimes^2})e_j\otimes e_k+P_k^{ij}(-\partial^{\otimes^2},1\otimes \partial)e_j^\ast\otimes e_k\\
&&-P_k^{ij}(1\otimes \partial,\partial\otimes 1)e_k\otimes
e_j^\ast-R_i^{kj}(\partial\otimes
1,-\partial^{\otimes^2})e_k\otimes
e_j-P_j^{ik}(-\partial^{\otimes^2},1\otimes
\partial)e_k^\ast\otimes e_j)=0.
\end{eqnarray*}
Similarly, for all $e_i^*\in A^{\ast c}$, we have
$$(I\otimes
L_A(e_i^\ast)_{-\partial^{\otimes^2}}-R_A(e_i^\ast)_{-\partial^{\otimes^2}}\otimes
I)(r+\tau r)=0.$$ Therefore (\ref{thq3}) holds. Hence by
Theorem~\ref{tt2}, $\Delta_{A\oplus A^{\ast c}}$ gives an ASI
conformal bialgebra structure on $A\oplus A^{\ast c}$.

Furthermore, for all $e_i\in A$, we have
\begin{eqnarray*}
\Delta_{A\oplus A^{\ast c}}(e_i)&=&\sum_{j=1}^n(I\otimes L_{A\oplus A^{\ast c}}(e_i)_\lambda-R_{A\oplus A^{\ast c}}(e_i)_\lambda\otimes I)(e_j\otimes e_j^\ast)|_{\lambda=-\partial^{\otimes^2}}\\
&=&\sum_{j=1}^n(e_j\otimes {e_i}_{-\partial^{\otimes^2}} e_j^\ast-{e_j}_{1\otimes \partial} e_i \otimes e_j^\ast)\\
&=&\sum_{j=1}^n(e_j\otimes (\sum_k P_j^{ki}(1\otimes \partial,\partial\otimes 1)e_k^\ast
+\sum_kR_i^{jk}(\partial\otimes 1,-\partial^{\otimes^2})e_k)\\
&&-\sum_k P_k^{ji}(1\otimes \partial, \partial\otimes 1)e_k\otimes e_j^\ast)\\
&=&\sum_{j=1}^n\sum_k R_i^{jk}(\partial\otimes
1,-\partial^{\otimes^2}) e_j\otimes e_k =\Delta_{A}(e_i).
\end{eqnarray*}
Similarly, for all $e_i^*\in A^{\ast c}$, we have
$$\Delta_{A\oplus A^{\ast c}}(e_i^\ast)=-\Delta_{A^{\ast
c}}(e_i^\ast).$$ Therefore $i_1: A\rightarrow A\oplus A^{\ast c}$
and $i_2: A^{\ast c}\rightarrow A\oplus A^{\ast c}$ are
homomorphisms of ASI conformal bialgebras. Hence the proof is
finished.
\end{proof}

\section{$\mathcal O$-operators of associative conformal algebras and dendriform conformal algebras}
We introduce the notion of $\mathcal O$-operators of associative conformal algebras to interpret the associative conformal Yang-Baxter equation.
In particular, an $\mathcal O$-operator of an associative conformal algebra $A$ associated to a bimodule gives an antisymmetric solution of associative conformal
Yang-Baxter equation in a semi-direct product associative conformal algebra. We also introduce the notion of dendriform conformal algebras to construct $\mathcal O$-operators of
their associated associative conformal algebras and hence give (antisymmetric) solutions of associative conformal Yang-Baxter equation.

Let $A$ be a finite associative conformal algebra which is free as
a $\mathbb{C}[\partial]$-module. Define a linear map $\varphi:
A\otimes A\rightarrow Chom(A^{\ast c}, A)$ as
\begin{eqnarray}
\varphi(u\otimes
v)_\lambda(g)=g_{-\lambda-\partial^A}(u)v,\;\;\forall\ u, v\in
A,g\in A^{\ast c}.
\end{eqnarray}
Here $\partial^A$ represents the action of $\partial$ on $A$.
Obviously,  $\varphi$ is a $\mathbb{C}[\partial]$-module
homomorphism. Similar to Proposition 6.1 in \cite{BKL}, we show
that $\varphi$ is a $\mathbb{C}[\partial]$-module isomorphism.


Set $r=\sum_i r_i\otimes l_i\in A\otimes A$. By $\varphi$, we
associate a conformal linear map $T^r\in Chom(A^{\ast c},A)$ given
by
\begin{eqnarray*}
T^r_\lambda(f)=\sum_if_{-\lambda-\partial^A}(r_i)l_i,\;\;\forall\
f\in A^{\ast c}.
\end{eqnarray*}

For all $f\in A^{\ast c}$ and $a\in A$, we define $\langle a, f\rangle_\lambda=\{f,a\}_{-\lambda}=f_{-\lambda}(a)$. Obviously,
\begin{eqnarray}\label{kh1}
\langle \partial a, f\rangle_\lambda=-\lambda\langle a, f\rangle_\lambda.
\end{eqnarray}
We also define
$$\langle a\otimes b\otimes c, f\otimes g\otimes h\rangle_{(\lambda, \nu,\theta)}=\langle a,f\rangle_\lambda \langle b,g\rangle_\nu\langle c,h\rangle_\theta,
\;\;\forall\ a,b,c\in A, f,g,h\in A^{\ast c}.
$$

By Proposition \ref{pr1}, we have
\begin{eqnarray}\label{x2}
\langle a_\lambda b, f\rangle_\mu=\langle b, L_A^\ast(a)_\lambda
f\rangle_{\mu-\lambda},~~~ \langle b_{\mu-\lambda}a,
f\rangle_\mu=\langle b, R_A^\ast(a)_\lambda
f\rangle_{\mu-\lambda},\;\;\forall a,b\in A, f\in A^{\ast c}.
\end{eqnarray}

\begin{thm}\label{th1}
Let $A$ be a finite associative conformal algebra which is free as
a $\mathbb{C}[\partial]$-module and $r\in A\otimes A$ be
antisymmetric. Then $r$ is a solution of associative conformal
Yang-Baxter equation if and only if $T^r\in Chom(A^{\ast c},A)$
satisfies
\begin{eqnarray}\label{x4}
T^r_{0}(f)_{\lambda}
T^r_{0}(g)-T_{0}^r(R_A^\ast(T_0^r(f))_{\lambda}g)-T^r_{0}(L_A^\ast(T^r_0(g))_{-\lambda-\partial}
f)=0,~~~~~\forall\ f,~g\in A^{\ast c}.
\end{eqnarray}
\end{thm}

\begin{proof}
Since $r$ is antisymmetric, we have $$T^r_\lambda(f)=\sum_i\langle
r_i, f\rangle_{\lambda+\partial^A}l_i=-\sum_i\langle l_i,
f\rangle_{\lambda+\partial^A}r_i,\; \;\forall\ f\in A^{\ast c}.$$
Obviously, the fact that $r\bullet r
~~\text{mod}~~(\partial^{\otimes^3}) = 0$ holds if and only if the
following equation holds:
$$\langle r\bullet r ~~\text{mod}~~(\partial^{\otimes^3}), f\otimes g\otimes h\rangle_{(\lambda,\eta,\nu)}=0,\;\;\forall\  f,g,h\in A^{\ast c},$$
if and only if the following equation holds:
\begin{eqnarray}\label{x1}
\langle r\bullet r, f\otimes g\otimes
h\rangle_{(\lambda,\eta,\nu)}=0 ~~~\text{mod}~~(\lambda+\eta+\nu),
\;\;\forall\ f,g,h\in A^{\ast c}.
\end{eqnarray}
Let $f$, $g$, $h\in A^{\ast c}$. Then we have
\begin{eqnarray*}
&&\langle \sum_{i,j}r_i\otimes r_j\otimes {l_i}_\mu l_j|_{\mu=\partial\otimes 1\otimes 1}, f\otimes g\otimes h \rangle_{(\lambda,\eta,\nu)}\\
&=&\sum_{i,j}\langle r_i, f\rangle_\lambda \langle r_j, g\rangle_\eta \langle {l_i}_{-\lambda}l_j, h \rangle_\nu=\langle (\sum_i\langle r_i, f\rangle_\lambda{l_i})_{-\lambda}(\sum_j\langle r_j, g\rangle_\eta l_j), h \rangle_\nu\\
&=&\langle T^r_{\lambda-\partial}(f)_{-\lambda} T^r_{\eta-\partial}(g),  h \rangle_\nu=\langle T^r_{0}(f)_{-\lambda} T^r_{\lambda+\eta+\nu}(g),  h \rangle_\nu,
\end{eqnarray*}
and
\begin{eqnarray*}
&&\langle \sum_{i,j}r_i\otimes {r_j}_\mu l_i\otimes l_j|_{\mu=-\partial^{\otimes^2}\otimes 1}, f\otimes g\otimes h \rangle_{(\lambda,\eta,\nu)}\\
&=&\sum_{i,j}\langle r_i,f\rangle_\lambda \langle{r_j}_{\lambda+\eta} l_i, g\rangle_\eta\langle l_j, h\rangle_\nu=\sum_{i,j}\langle{r_j}_{\lambda+\eta} (\langle r_i,f\rangle_\lambda l_i), g\rangle_\eta\langle l_j, h\rangle_\nu\\
&=&\sum_j\langle{r_j}_{\lambda+\eta} (T_{\lambda-\partial}^r(f)), g \rangle_\eta\langle l_j, h\rangle_\nu=\sum_j \langle{r_j}, R_A^\ast(T_0^r(f))_{-\lambda}g\rangle_{\lambda+\eta}\langle l_j, h\rangle_\nu\\
&=&\langle (\sum_j\langle{r_j}, R_A^\ast(T_0^r(f))_{-\lambda}g\rangle_{\lambda+\eta}l_j), h\rangle_\nu
=\langle T_{\lambda+\eta-\partial}^r(R_A^\ast(T_0^r(f))_{-\lambda}g), h\rangle_\nu\\
&=&\langle T_{\lambda+\eta+\nu}^r(R_A^\ast(T_0^r(f))_{-\lambda}g), h\rangle_\nu.
\end{eqnarray*}
Similarly, we have
\begin{eqnarray*}
\langle \sum_{i,j} {r_i}_\mu r_j\otimes l_i\otimes l_j|_{\mu=1\otimes \partial\otimes 1}, f\otimes g\otimes h\rangle_{(\lambda,\eta,\nu)}
=-\langle T^r_{\lambda+\eta+\nu}(L_A^\ast(T^r_0(g))_{-\eta} f), h\rangle_\nu.
\end{eqnarray*}
Therefore (\ref{x1}) holds if and only if
\begin{eqnarray*}
&&\langle T^r_{0}(f)_{-\lambda}
T^r_{\lambda+\eta+\nu}(g)-T_{\lambda+\eta+\nu}^r(R_A^\ast(T_0^r(f))_{-\lambda}g)
-T^r_{\lambda+\eta+\nu}(L_A^\ast(T^r_0(g))_{-\eta} f),  h
\rangle_\nu\\&&=0 ~~~\text{mod} ~~~(\lambda+\eta+\nu),\;\;\forall\
f,g,h\in A^{\ast c}.
\end{eqnarray*}
 It is straightforward that
the above equality holds if and only if
\begin{eqnarray}\label{x2}
T^r_{0}(f)_{-\lambda}
T^r_{0}(g)-T_{0}^r(R_A^\ast(T_0^r(f))_{-\lambda}g)-T^r_{0}(L_A^\ast(T^r_0(g))_{\lambda-\partial}
f)=0,\;\;\forall f,g\in A^{\ast c}.
\end{eqnarray}
Therefore the conclusion follows by replacing $\lambda$ by
$-\lambda$ in (\ref{x2}).
\end{proof}
\begin{cor}\label{corhomo}
Let $A$ be a finite associative conformal algebra which is free as
a $\mathbb{C}[\partial]$-module, $r\in A\otimes A$ be
antisymmetric and $\Delta_A$ be the map  defined by (\ref{eQ1})
through $r$. Suppose the associative conformal algebra structure
on $A^{\ast c}$ is obtained from $\Delta_A$. Let $T^r\in
\text{Chom}(A^{\ast c},A)$ be the element corresponding to $r$
through the isomorphism $A\otimes A\cong \text{Chom}(A^{\ast
c},A)$. Then $T_0^r: A^{\ast c}\rightarrow A$ is a homomorphism of
associative conformal algebras.
\end{cor}
\begin{proof}
Let $\{e_1,\cdots, e_n\}$ be a $\mathbb{C}[\partial]$-basis of $A$
and $\{e_1^{\ast},\cdots, e_n^{\ast}\}$ be the dual
$\mathbb{C}[\partial]$-basis. Set $${e_i}_\lambda
e_j=\sum_kP_{ij}^ke_k,\;\;r=\sum_{i,j}a_{ij}(\partial\otimes
1,1\otimes \partial)e_i \otimes e_j,$$ where
$a_{ij}(\lambda,\mu)\in \mathbb{C}[\lambda,\mu]$. Then
\begin{eqnarray*}
\Delta_A(e_k)&=&(I\otimes L_A(e_k)_\lambda-R_A(a)_\lambda\otimes 1)r|_{\lambda=-\partial^{\otimes^2}}\\
&=&\sum_{i,j}a_{ij}(\partial\otimes 1,\lambda+1\otimes \partial)e_i \otimes {e_k}_\lambda e_j
-\sum_{i,j}a_{ij}(\lambda+\partial\otimes 1,1\otimes \partial){e_i}_{-\lambda-\partial}e_k\otimes e_j\\
&=&\sum_{i,j,s}(a_{ij}(\partial\otimes 1,-\partial\otimes 1)P_{kj}^s(-\partial^{\otimes^2},1\otimes \partial)-a_{js}(-1\otimes \partial, 1\otimes \partial)P_{jk}^i(1\otimes \partial, \partial\otimes 1))e_i\otimes e_s.
\end{eqnarray*}
Note that $T_0^r(e_k^\ast)=\sum_ja_{kj}(-\partial,\partial)e_j$.
Then by Proposition \ref{Dpro1},
\begin{eqnarray*}
{e_l^\ast}_\lambda e_t^\ast&=&\sum_{j,k}(a_{lj}(\lambda,-\lambda)P_{kj}^t(\partial,-\lambda-\partial)
-a_{jt}(\lambda+\partial,-\lambda-\partial)P_{jk}^l(-\lambda-\partial,\lambda))e_k^\ast\\
&=&\sum_{j,k}(a_{lj}(\lambda,-\lambda) {e_t^\ast}_{-\lambda-\partial^{A^{\ast c}}}({e_k}_{\partial^{A^{\ast c}}} e_j)-a_{jt}(\lambda+\partial,-\lambda-\partial)
{e_l^\ast}_\lambda({e_j}_{-\lambda-\partial^{A^{\ast c}}}e_k))e_k^\ast\\
&=&\sum_{j,k}(a_{lj}(\lambda,-\lambda) {e_t^\ast}_{-\lambda-\partial^{A^{\ast c}}}({e_k}_{\partial^{A^{\ast c}}} e_j)+a_{tj}(-\lambda-\partial,\lambda+\partial)
{e_l^\ast}_\lambda({e_j}_{-\lambda-\partial^{A^{\ast c}}}e_k))e_k^\ast\\
&=&\sum_{k}({e_t^\ast}_{-\lambda-\partial^{A^{\ast c}}}({e_k}_{\partial^{A^{\ast c}}} T_0^r(e_l))
+{e_l^\ast}_\lambda(T_0^r({e_t})_{-\lambda-\partial^{A^{\ast c}}}e_k))e_k^\ast\\
&=&R_A^{\ast}(T_0^r(e_l))_\lambda e_t^\ast+L_A^\ast(T_0^r(e_t))_{-\lambda-\partial} e_l^\ast.
\end{eqnarray*}
Thus by Theorem \ref{th1}, we have
$${T_0^r(e_l^\ast)}_\lambda{T_0^r(e_t^\ast)}=T_0^r({e_l^\ast}_\lambda
e_t^\ast),\;\; \forall\ l, t\in\{1,\cdots, n\}.$$ Therefore
$T_0^r: A^{\ast c}\rightarrow A$ is a homomorphism of associative
conformal algebras.
\end{proof}

Let $A$ be a Frobenius conformal algebra with a non-degenerate
invariant  conformal bilinear form $\langle \cdot , \cdot
\rangle_\lambda$, which is finitely generated and free as a
$\mathbb{C}[\partial]$-module. Define
\begin{eqnarray}
\langle a\otimes b, c\otimes d\rangle_{(\lambda,\mu)} =\langle a,
c\rangle_\lambda \langle b, d\rangle_\mu,\;\;\forall\ a,b,c,d\in
A.
\end{eqnarray}
Let $r=\sum_{i} r_i\otimes l_i\in A\otimes A$. Define a linear map $P^r: A\rightarrow A[\lambda]$ by
\begin{eqnarray*}
\langle r, u\otimes v\rangle_{(\lambda,\mu)}=\langle P_{\lambda-\partial}^r(u), v\rangle_\mu,\;\;\forall u,v\in A.
\end{eqnarray*}
Obviously, $P^r\in \text{Cend}(A)$.

\begin{cor}\label{th2}
Let $A$ be a symmetric Frobenius conformal algebra which  is finitely generated and free as a $\mathbb{C}[\partial]$-module and $r\in A\otimes A$ be antisymmetric. Then $r$ is a solution of
associative conformal Yang-Baxter equation in $A$ if and only if $P^r \in \text{Cend}(A)$ satisfies
\begin{eqnarray}\label{Rota1}
P^r_{0}(a)_{\lambda}P^r_{0}(b)=P_{0}^r(P_0^r(a)_{\lambda}b)+P^r_{0}(
a_\lambda P^r_0(b)), \;\;\forall\ a,b\in A.
\end{eqnarray}
\end{cor}

\begin{proof}
Since $A$ has a non-degenerate symmetric invariant conformal
bilinear form,   $\varphi: A\longrightarrow
A^{\ast c},~~ a\mapsto \varphi_a$ defined by
$$(\varphi_a)_\lambda b=\langle a,b\rangle_\lambda,\quad \forall\ a, b\in A,$$
is an isomorphism of $\mathbb{C}[\partial]$-modules. By the definitions of $T^r$ in Theorem \ref{th1} and $P^r$, we get
$P^r=T^r\circ \varphi$. Therefore $P^r_0=T^r_0\circ \varphi$. Since $\varphi$ is a $\mathbb{C}[\partial]$-module
isomorphism, for any $f$, $g\in A^{\ast c}$, we assume that $\varphi(a)=f$ and $\varphi(b)=g$.
Then (\ref{x4}) becomes
\begin{eqnarray}
T^r_{0}(\varphi(a))_{\lambda}
T^r_{0}(\varphi(b))-T_{0}^r(R_A^\ast(T_0^r(\varphi(a)))_{\lambda}\varphi(b))-T^r_{0}(L_A^\ast(T^r_0(\varphi(b)))_{-\lambda-\partial}
\varphi(a))=0.
\end{eqnarray}
Thus
\begin{eqnarray}\label{RR1}
P^r_{0}(a)_{\lambda}
P^r_{0}(b)-T_{0}^r(R_A^\ast(P_0^r(a))_{\lambda}\varphi(b))-T^r_{0}(L_A^\ast(P^r_0(b))_{-\lambda-\partial}
\varphi(a))=0,\;\;\forall\ a,b\in A.
\end{eqnarray}
For all $a,b,c\in A$, we have
\begin{eqnarray*}
(R_A^\ast(P_0^r(a))_{\lambda}\varphi(b))_\mu c&=&\varphi(b)_{\mu-\lambda}(R_A(P_0^r(a))_\lambda c)=\langle b, R_A(P_0^r(a))_\lambda c\rangle_{\mu-\lambda}\\
&=&\langle b, c_{-\lambda-\partial}P_0^r(a)\rangle_{\mu-\lambda}=\langle b_{\mu-\lambda}c, P_0^r(a)\rangle_{-\lambda}\\
&=&\langle P_0^r(a),b_{\mu-\lambda}c\rangle_{\lambda}= \langle
P_0^r(a)_\lambda b, c\rangle_{\mu}.
\end{eqnarray*}
Therefore we have
$$R_A^\ast(P_0^r(a))_{\lambda}\varphi(b)=\varphi(P_0^r(a)_\lambda
b),\;\;\forall\ a,b\in A.$$ Similarly, we have
$$L_A^\ast(P^r_0(b))_{-\lambda-\partial}
\varphi(a)=\varphi(a_\lambda P_0^r(b)),\;\;\forall\ a,b\in A.$$
Hence (\ref{Rota1}) follows from (\ref{RR1}).

Similarly, we also obtain (\ref{x4}) from (\ref{Rota1}) through
$\varphi$. Then by Theorem \ref{th1}, the conclusion holds.
\end{proof}

By the fact that the $T_0^r$ in Theorems \ref{th1} and the $P_0^r$
in Corollary~\ref{th2} are $\mathbb{C}[\partial]$-module
homomorphisms, we present the following notions.

\begin{defi}{\rm
Let $A$ be an associative conformal algebra and $(M, l_A, r_A)$ be a bimodule of $A$.
A $\mathbb{C}[\partial]$-module homomorphism $T: M\rightarrow A$ is called an {\bf $\mathcal{O}$-operator} associated with
$(M, l_A, r_A)$ if $T$ satisfies
\begin{eqnarray}
T(u)_{\lambda}T(v)=T(l_A(T(u))_{\lambda}v)+T(r_A(T(v))_{-\lambda-\partial} u), \;\;\forall \text{$u$, $v\in M$.}
\end{eqnarray}
In particular, an $\mathcal{O}$-operator $T:A\rightarrow A$
associated with the bimodule $(A, L_A, R_A)$ is called  a {\bf
Rota-Baxter operator (of weight zero)} on $A$, that is, $T$ is a
$\mathbb{C}[\partial]$-module homomorphism satisfying
\begin{eqnarray}
T(a)_{\lambda}T(b)=T(T(a)_{\lambda}b)+T(a_\lambda T(b)), \;\;\forall\ \text{$a$, $b\in A$.}
\end{eqnarray}
}\end{defi}

\begin{ex}
Let $A$ be a finite associative conformal algebra which is free as
a $\mathbb{C}[\partial]$-module and $r\in A\otimes A$ be
antisymmetric. Then $r$ is a solution of associative conformal
Yang-Baxter equation if and only if $T_0^r$ is an
$\mathcal{O}$-operator associated with the bimodule $(A^{\ast c},
R_A^\ast, L_A^\ast)$. If in addition, $A$ is a symmetric Frobenius
conformal algebra, that is, $A$ has a non-degenerate symmetric
invariant conformal bilinear form, then $r$ is a solution of
associative conformal Yang-Baxter equation if and only if $P_0^r$
is a Rota-Baxter operator (of weight zero) on $A$.
\end{ex}
\begin{ex}
Let $A$ be an associative conformal algebra. The identity map $I$ is an $\mathcal{O}$-operator associated with the bimodule $(A, L_A,0)$ or $(A,0, R_A)$.
\end{ex}

Let $(M, l_A, r_A)$ be a bimodule of an associative conformal
algebra $A$. Then $(M^{\ast c}, r_A^\ast, l_A^\ast)$ is a bimodule
of $A$ by Proposition \ref{pr1}. Suppose that $M$ is a
$\mathbb{C}[\partial]$-module of finite rank. By Proposition 6.1
in \cite{BKL}, $M^{\ast c}\otimes A\cong Chom(M,A)$ as
$\mathbb{C}[\partial]$-modules through the isomorphism $\varphi$
defined as $$\varphi(f\otimes a)_\lambda
v=f_{\lambda+\partial^A}(v)a,\;\;\forall\ a\in A, v\in M, f\in
M^{\ast c}.$$ By the $\mathbb{C}[\partial]$-module actions on
$M^{\ast c}\otimes A$, we also get $M^{\ast c}\otimes A\cong
A\otimes M^{\ast c}$ as $\mathbb{C}[\partial]$-modules. Therefore
as $\mathbb{C}[\partial]$-modules, $Chom(M,A)\cong A\otimes
M^{\ast c}$. Consequently, for any $T\in Chom(M,A)$, we associate
an $r_T\in A\otimes M^{\ast c}\subset (A\ltimes_{r_A^\ast,
l_A^\ast} M^{\ast c})\otimes (A\ltimes_{r_A^\ast, l_A^\ast}
M^{\ast c})$.

\begin{thm}\label{thh2}
Let $A$ be a finite associative conformal algebra and $(M, l_A,
r_A)$ be a finite bimodule of $A$. Suppose that $A$ and $M$ are
free as $\mathbb{C}[\partial]$-modules. Let $T\in Chom(M,A)$ and
$r_T\in A\otimes M^{\ast c}\subset (A\ltimes_{r_A^\ast, l_A^\ast}
M^{\ast c})\otimes (A\ltimes_{r_A^\ast, l_A^\ast} M^{\ast c})$ be
the  element corresponding to $T$ under the above correspondence.
Then $r=r_T-\tau r_T$ is an antisymmetric solution of the
associative conformal Yang-Baxter equation in $A\ltimes_{r_A^\ast,
l_A^\ast} M^{\ast c}$ if and only if $T_0=T_\lambda|_{\lambda=0}$
is an $\mathcal{O}$-operator associated with the bimodule $(M,
l_A, r_A)$.
\end{thm}
\begin{proof}
Let $\{e_1,\cdots, e_n\}$ be a $\mathbb{C}[\partial]$-basis of
$A$,  $\{v_1,\cdots, v_m\}$ be a $\mathbb{C}[\partial]$-basis of
$M$ and $\{v_1^\ast,\cdots, v_m^\ast\}$ be the dual
$\mathbb{C}[\partial]$-basis of $M^{\ast c}$. Assume that
$$T_\lambda(v_i)=\sum_{j=1}^n
g_{ij}(\lambda,\partial)e_j,\;\;\forall\ i=1, \cdots, m,$$ where
$g_{ij}(\lambda,\partial)\in \mathbb{C}[\lambda,\partial]$. Then
we have
\begin{eqnarray*}
r_T=\sum_{j=1}^n\sum_{i=1}^mg_{ij}(-\partial^{\otimes^2},\partial\otimes 1)e_j\otimes v_i^\ast.
\end{eqnarray*}
Therefore we have
\begin{eqnarray*}
r=\sum_{i,j}(g_{ij}(-\partial^{\otimes^2},\partial\otimes 1)e_j\otimes v_i^\ast-g_{ij}(-\partial^{\otimes^2}, 1\otimes \partial)v_i^\ast \otimes e_j).
\end{eqnarray*}
Moreover, by the definition of $(M^{\ast c}, r_A^\ast, l_A^\ast)$,
we have
\begin{eqnarray*}
l_A^\ast(e_i)_\lambda v_j^\ast=\sum_k {v_j^\ast}_{-\lambda-\partial}(l_A(e_i)_\lambda v_k) v_k^\ast,~~~
r_A^\ast(e_i)_\lambda v_j^\ast=\sum_k {v_j^\ast}_{-\lambda-\partial}(r_A(e_i)_\lambda v_k) v_k^\ast.
\end{eqnarray*}
Then we get
\begin{eqnarray*}
r\bullet r&\equiv& \sum_{i,j,k,l}(-g_{ij}(0,\partial\otimes 1\otimes 1)g_{kl}(0,-1\otimes \partial\otimes 1)e_j\otimes v_k^\ast \otimes {v_i^\ast}_\mu e_l|_{\mu=\partial\otimes 1\otimes 1}\\
&&-g_{ij}(0,-\partial\otimes 1\otimes 1)g_{kl}(0,1\otimes \partial\otimes 1)v_i^\ast \otimes e_l\otimes {e_j}_\mu v_k^\ast|_{\mu=\partial\otimes 1\otimes 1}\\
&&+g_{ij}(0,-\partial\otimes 1\otimes 1)g_{kl}(0,-1\otimes \partial\otimes 1)v_i^\ast \otimes v_k^\ast \otimes {e_j}_\mu e_l|_{\mu=\partial\otimes 1\otimes 1}\\
&&-g_{ij}(0,\partial\otimes 1\otimes 1)g_{kl}(0,-1\otimes 1\otimes \partial)e_j\otimes {e_l}_\mu v_i^\ast\otimes v_k^\ast|_{\mu=-\partial^{\otimes^2}\otimes 1}\\
&&+g_{ij}(0,-\partial\otimes 1\otimes 1)g_{kl}(0,-1\otimes 1\otimes \partial)v_i^\ast\otimes {e_l}_\mu e_j\otimes v_k^\ast|_{\mu=-\partial^{\otimes^2}\otimes 1}\\
&&-g_{ij}(0,-\partial\otimes 1\otimes 1)g_{kl}(0,1\otimes 1\otimes \partial) v_i^\ast \otimes
{v_k^\ast}_\mu e_j\otimes e_l|_{\mu=-\partial^{\otimes^2}\otimes 1}\\
&&+g_{ij}(0,-1\otimes \partial\otimes 1)g_{kl}(0,-1\otimes 1\otimes \partial)
{e_j}_\mu e_l \otimes v_i^\ast\otimes v_k^\ast|_{\mu=1\otimes \partial \otimes 1}\\
&&-g_{ij}(0,-1\otimes \partial\otimes 1)g_{kl}(0,1\otimes 1\otimes \partial)
{e_j}_\mu v_k^\ast \otimes v_i^\ast\otimes e_l|_{\mu=1\otimes \partial \otimes 1}\\
&&-g_{ij}(0,1\otimes \partial\otimes 1)g_{kl}(0,-1\otimes 1\otimes \partial)
{v_i^\ast}_\mu e_l \otimes e_j\otimes v_k^\ast|_{\mu=1\otimes \partial \otimes 1}
)~~~\text{$mod~~(\partial^{\otimes^3})$}\\
&\equiv & \sum_{i,k}((-T_0(v_i)\otimes v_k^\ast\otimes
{{v_i}^\ast}_\mu T_0(v_k)
-v_i^\ast \otimes T_0(v_k)\otimes T_0(v_i)_\mu v_k^\ast\\
&&+v_i^\ast \otimes v_k^\ast \otimes {T_0(v_i)}_\mu T_0(v_k))|_{\mu=\partial\otimes 1 \otimes1}\\
&&+(-T_0(v_i)\otimes T_0(v_k)_\mu v_i^\ast\otimes v_k^\ast+v_i^\ast \otimes {T_0(v_k)}_\mu T_0(v_i) \otimes v_k^\ast\\
&&-v_i^\ast\otimes {v_k^\ast}_\mu T_0(v_i)\otimes T_0(v_k))|_{\mu=1\otimes 1 \otimes \partial}\\
&&+(T_0(v_i)_\mu T_0(v_k) \otimes v_i^\ast \otimes v_k^\ast-T_0(v_i)_\mu v_k^\ast \otimes v_i^\ast \otimes T_0(v_k)\\
&&-{v_i^\ast}_\mu T_0(v_k)\otimes T_0(v_i)\otimes v_k^\ast) |_{\mu=1\otimes \partial \otimes 1}~~~\text{$mod~~(\partial^{\otimes^3})$}.\end{eqnarray*}
Since $T_0$ is a $\mathbb{C}[\partial]$-module homomorphism, we have
\begin{eqnarray*}
&&\sum_{i,k}T_0(v_i)\otimes v_k^\ast\otimes {v_i^\ast}_\mu T_0(v_k)|_{\mu=\partial\otimes 1\otimes 1}\\
&=&\sum_{i,k}T_0(v_i)\otimes v_k^\ast\otimes l_A^\ast(T_0(v_k))_{-\mu-\partial}v_i^\ast|_{\mu=\partial\otimes 1\otimes 1}\\
&\equiv &\sum_{i,k}T_0(v_i)\otimes v_k^\ast\otimes l_A^\ast(T_0(v_k))_{1\otimes \partial\otimes 1}v_i^\ast ~~\text{$mod~~(\partial^{\otimes^3})$}\\
&\equiv &\sum_{i,j,k}T_0(v_i)\otimes v_k^\ast\otimes {v_i^\ast}_{\partial\otimes 1\otimes 1}(l_A(T_0(v_k))_{1\otimes \partial\otimes 1}v_j)v_j^\ast ~~\text{$mod~~(\partial^{\otimes^3})$}\\
&\equiv &\sum_{i,j,k}T_0({v_i^\ast}_{\partial}(l_A(T_0(v_k))_{1\otimes \partial\otimes 1}v_j)v_i)\otimes v_k^\ast\otimes v_j^\ast ~~\text{$mod~~(\partial^{\otimes^3})$}\\
&\equiv & \sum_{j,k}T_0(l_A((T_0(v_k))_{\mu}v_j)\otimes v_k^\ast\otimes v_j^\ast|_{\mu= 1\otimes \partial\otimes 1} ~~\text{$mod~~(\partial^{\otimes^3})$}.
\end{eqnarray*}
Similarly, we have
\begin{eqnarray*}
&&r\bullet r ~~\text{$mod~~(\partial^{\otimes^3})$}\\
&\equiv& \sum_{i,k}(({T_0(v_k)}_\mu T_0(v_i)-T_0(l_A(T(v_k))_\mu v_i)-T_0(r_A(T_0(v_i))_{-\mu-\partial} v_k)  \otimes v_k^\ast\otimes v_i^\ast
)|_{\mu=1\otimes \partial\otimes 1}\\
&&+(v_i^\ast \otimes ({T_0(v_k)}_\mu T_0(v_i)-T_0(l_A(T(v_k))_\mu v_i)-T_0(r_A(T_0(v_i))_{-\mu-\partial} v_k) \otimes v_k^\ast)|_{\mu=1\otimes 1\otimes \partial}\\
&&+(v_k^\ast\otimes v_i^\ast\otimes ({T_0(v_k)}_\mu T_0(v_i)-T_0(l_A(T(v_k))_\mu v_i)-T_0(r_A(T_0(v_i))_{-\mu-\partial} v_k) \otimes v_k^\ast)|_{\mu=\partial\otimes 1\otimes 1}.
\end{eqnarray*}
Therefore $r$ is a solution of associative conformal Yang-Baxter
equation in the associative conformal algebra $A\ltimes_{r_A^\ast,
l_A^\ast} M^{\ast c}$ if and only if
\begin{eqnarray*}
{T_0(v_k)}_\mu T_0(v_i)=T_0(l_A(T(v_k))_\mu
v_i)+T_0(r_A(T_0(v_i))_{-\mu-\partial} v_k), \;\;\forall\ i, k\in
\{1,\cdots, m\}.
\end{eqnarray*}
Thus this conclusion holds.
\end{proof}

At the end of this paper, we introduce a class of conformal algebras, namely, dendriform conformal algebras,
which are used to construct $\mathcal O$-operators naturally and hence give solutions of associative conformal Yang-Baxter equation.

\begin{defi}{\rm
Let $A$ be a $\mathbb{C}[\partial]$-module with two bilinear products $\prec_\lambda$ and $\succ_\lambda: A\times A\rightarrow A[\lambda]$. If for all $a$, $b$, $c\in A$,
\begin{eqnarray}
(\partial a)\succ_\lambda b=-\lambda a\succ_\lambda b, ~~a\succ_\lambda (\partial b)=(\partial+\lambda) (a\succ_\lambda b),\\
(\partial a)\prec_\lambda b=-\lambda a\prec_\lambda b, ~~a\prec_\lambda (\partial b)=(\partial+\lambda) (a\prec_\lambda b),\\
(a\prec_\lambda b)\prec_{\lambda+\mu} c=a\prec_\lambda(b\ast_\mu c),\\~~(a\succ_\lambda b)\prec_{\lambda+\mu} c=a\succ_\lambda(b\prec_\mu c),\\
~~a\succ_\lambda(b\succ_\mu c)=(a\ast_\lambda b)\succ_{\lambda+\mu}c,
\end{eqnarray}
where $a\ast_\lambda b=a\prec_\lambda b+a\succ_\lambda b$, then $(A,\prec_\lambda,\succ_\lambda)$ is called a {\bf dendriform conformal algebra}.
}\end{defi}

\begin{rmk}{\rm
It is obvious that $(A,\prec_\lambda,\succ_\lambda)$ with
$\succ_\lambda$ being trivial (or $\prec_\lambda$ being trivial)
is a dendriform conformal algebra if and only if
$(A,\prec_\lambda)$ (or $(A,\succ_\lambda))$ is an associative
conformal algebra.}
\end{rmk}

\begin{ex}
Let $(A,\prec,\succ)$ be a dendriform algebra {\rm (\cite{Lod})}.
Then $\text{Cur}(A)=\mathbb{C}[\partial]\otimes A$ is endowed a
natural dendriform conformal algebra as follows.
\begin{eqnarray}
a\prec_\lambda b=a\prec b,~~~a\succ_\lambda b=a\succ b,~~~\forall
a, b\in A.
\end{eqnarray}
Moreover, $(\text{Cur}(A), \prec_\lambda, \succ_\lambda)$ is
called a {\bf current dendriform conformal algebra}. It is
straightforward that any dendriform conformal algebra which is
free and of rank one as a $\mathbb{C}[\partial]$-module is
current.
\end{ex}

\begin{pro}
Let $(A,\prec_\lambda,\succ_\lambda)$ be a dendriform conformal algebra. Define
\begin{equation}a\ast_\lambda b=a\prec_\lambda b+a\succ_\lambda b,\;\;\forall\ a,b\in A.\end{equation}
Then $(A,\ast_\lambda)$ is an associative conformal algebra. We call $(A,\ast_\lambda)$ {\bf the associated associative conformal algebra of $(A,\prec_\lambda,\succ_\lambda)$ } and $(A,\prec_\lambda,\succ_\lambda)$ is called a {\bf compatible dendriform conformal algebra structure on the associative conformal algebra $(A,\ast_\lambda)$}.
\end{pro}
\begin{proof}
It is straightforward.
\end{proof}

Let $(A,\prec_\lambda,\succ_\lambda)$ be a dendriform conformal
algebra. Set \begin{eqnarray*}
&&{L_\succ(a)}_\lambda(b)=a\succ_\lambda
b,\;\; {L_\prec(a)}_\lambda(b)=a\prec_\lambda b,\\
&&{R_\succ(a)}_\lambda b=b\succ_{-\lambda-\partial}
a,\;\;{R_\prec(a)}_\lambda b=b\prec_{-\lambda-\partial} a,\;\;
\forall\  a, b\in A.
\end{eqnarray*}

\begin{pro}\label{pro:bim}
Let $(A,\prec_\lambda,\succ_\lambda)$ be a dendriform conformal
algebra. Then $(A, L_\succ, R_\prec)$ is a bimodule of the
associated associative conformal algebra $(A,\ast_\lambda)$. Hence
the identity $I$ is an $\mathcal{O}$-operator of the associative
conformal algebra $(A,\ast_\lambda)$ associated with $(A, L_\succ,
R_\prec)$.
\end{pro}

\begin{proof}
It is straightforward.
\end{proof}

\begin{pro}\label{cor1}
Let $A$ be an associative conformal algebra and $(M,l_A,r_A)$ be a bimodule of $A$. Suppose $T: M\rightarrow A$ is an $\mathcal{O}$-operator associated with $(M,l_A,r_A)$.
Then  the following $\lambda$-product
\begin{eqnarray}
u\succ_\lambda v=l_A(T(u))_\lambda v,~~~u\prec_\lambda v=r_A(T(v))_{-\lambda-\partial}u,~~~~~u,~v\in M,
\end{eqnarray}
endows a dendriform conformal algebra structure on $M$. Therefore
there is an associated associative conformal algebra structure on
$M$ and $T:M\rightarrow A$ is a homomorphism of associative
conformal algebras. Moreover, $T(M)\subset A$ is an associative
conformal subalgebra of $A$ and there is also a dendriform
conformal algebra structure on $T(M)$ defined by
\begin{eqnarray}\label{ww1}
T(u)\succ_\lambda T(v)=T(u\succ_\lambda v),~~T(u)\prec_\lambda T(v)=T(u\prec_\lambda v),~~~\text{$u$, $v\in M$.}
\end{eqnarray}
Furthermore, the associated associative conformal algebra on
$T(M)$ is a subalgebra of $A$ and $T: M\rightarrow A$ is a
homomorphism of dendriform conformal algebras.
\end{pro}

\begin{proof}
For all $u$, $v$, $w\in M$,  we have
\begin{eqnarray*}
&&(u\prec_\lambda v)\prec_{\lambda+\mu}w-u\prec_\lambda(v\prec_\mu w+v\succ_\mu w)\\
&=&r_A(T(w))_{-\lambda-\mu-\partial}(r_A(T(v))_{-\lambda-\partial}u)-r_A(r_A(T(w))_{-\mu-\partial}v)_{-\lambda-\partial}u
-r_A(l_A(T(v))_\mu w)_{-\lambda-\partial}u\\
&=&r_A(T(v)_\mu T(w))_{-\lambda-\partial}u-r_A(r_A(T(w))_{-\mu-\partial}v)_{-\lambda-\partial}u
-r_A(l_A(T(v))_\mu w)_{-\lambda-\partial}u\\
&=&r_A(T(v)_\mu T(w)-r_A(T(w))_{-\mu-\partial}v-l_A(T(v))_\mu w)_{-\lambda-\partial}u=0.
\end{eqnarray*}
Similarly, we have $$u\succ_\lambda(v\succ_\mu w)=(u\succ_\lambda
v+u\prec_\lambda v)\succ_{\lambda+\mu}w,\;\;\forall\ u,v,w\in M.$$
Moreover, since $(M,l_A,r_A)$ is a bimodule of $A$, we have
$$(u\succ_\lambda v)\prec_{\lambda+\mu}w=u\succ_\lambda(v\prec_\mu
w),\;\;\forall\ u,v,w\in M.$$ Hence $(M,
\prec_\lambda,\succ_\lambda)$  is a dendriform conformal algebra.
Moreover, the other conclusions follow straightforwardly.
Therefore the conclusion holds.
\end{proof}

\begin{cor}\label{ll1}
Let $(A, \ast_\lambda)$ be an associative conformal algebra. There is a compatible dendriform conformal algebra structure on $A$ if and only if there exists a bijective $\mathcal{O}$-operator $T: M\rightarrow A$  associated with some bimodule $(M, l_A, r_A)$ of $A$.
\end{cor}

\begin{proof}
Suppose that there is a compatible dendriform conformal algebra structure
$(A,\succ_\lambda,\prec_\lambda)$ on $A$. Then by Proposition~\ref{pro:bim}, the identity map $I: A\rightarrow A$ is a bijective $\mathcal{O}$-operator of $A$ associated with $(A, L_\succ, R_\prec)$.

Conversely, suppose that there exists a bijective $\mathcal{O}$-operator $T: M\rightarrow A$ of $(A,\ast_\lambda)$ associated with a bimodule $(M,l_A, r_A)$. Then by Proposition~\ref{cor1} with a straightforward checking, we have
\begin{eqnarray}
a\succ_\lambda b=T(l_A(a)_\lambda T^{-1}(b)),~~~a\prec_\lambda
b=T({r_A(b)}_{-\lambda-\partial}T^{-1}(a)),\;\;\forall\ a,b\in A,
\end{eqnarray}
defines a compatible dendriform conformal algebra structure on $A$.
\end{proof}

Finally, there is a construction of (antisymmetric) solutions of associative conformal Yang-Baxter equation from dendriform conformal algebras.

\begin{thm}\label{the1}
Let $(A, \succ_\lambda, \prec_\lambda)$ be a finite dendriform conformal algebra which is free
as a $\mathbb{C}[\partial]$-module. Then
\begin{eqnarray}\label{eqn4}
r=\sum_{i=1}^n(e_i\otimes e_i^\ast-e_i^\ast\otimes e_i)
\end{eqnarray}
is a solution of associative conformal Yang-Baxter equation in the
associative conformal algebra $A\ltimes_{R_A^\ast, L_A^\ast}
A^{\ast c}$, where $\{e_1,\cdots,e_n\}$ is a $\mathbb
C[\partial]$-basis of $A$ and $\{e_1^\ast,\cdots, e_n^\ast\}$ is
the dual $\mathbb C[\partial]$-basis of $A^{\ast c}$.
\end{thm}

\begin{proof}
By Proposition ~\ref{pro:bim}, $T=I: A\rightarrow A $ is an $\mathcal{O}$-operator associated with $(A, L_\succ, R_\prec)$.
Then by Theorem \ref{thh2}, the conclusion holds.
\end{proof}


{\bf Acknowledgments}
{This work was supported by the National Natural Science Foundation of China
(11425104, 11931009), the Zhejiang Provincial Natural Science Foundation of China
(LY20A010022) and the Scientific Research Foundation of Hangzhou Normal University (2019QDL012).
C. Bai is also supported by the Fundamental Research Funds for the Central Universities and Nankai ZhiDe
Foundation.}


\begin{thebibliography}{99}
\bibitem{A1}M. Aguiar, On the associative analog of Lie bialgebras, J. Algebra {\bf 244} (2001), 492-532.
\bibitem{Bai1}C. Bai, Double constructions of Frobenius algebras, Connes cocycles and their duality, J. Noncommut. Geom. {\bf 4} (2010), 475-530.
\bibitem{BDK} A. Barakat, A.De Sole, V. Kac, Poisson vertex algebras in the theory of Hamiltonian equations, Japan. J. Math. {\bf 4} (2009), 141-252.
\bibitem{BFK1}L. Bokut, Y. Fong, W. Ke, Free associative conformal algebras, Proc. of the 2nd Tainan-Moscow Algebra and Combinatorics Workshop, Tainan, 1997, 13-25.
\bibitem{BFK2}L. Bokut, Y. Fong, W. Ke, Grobner-Shirshov bases and composition lemma for associative conformal algebras: an example, Contemp. Math. {\bf 264} (2000), 63-90.
\bibitem{BFK3}L. Bokut, Y. Fong, W. Ke, Composition-Diamond lemma for associative conformal algebras, J. Algebra {\bf 272} (2004), 739-774.
\bibitem{BKL}C. Boyallian, V. Kac, J. Liberati, On the classification of subalgebras of $Cend_N$ and $gc_N$, J. Algebra {\bf 260} (2003), 32-63.
\bibitem{BKV} B. Bakalov, V. Kac, A. Voronov, Cohomology of conformal algebras,  Comm. Math. Phys. {\bf 200} (1999), 561-598.
\bibitem{DK1} A. D'Andrea, V. Kac, Structure theory of
finite conformal algebras, Selecta Math. New Ser. {\bf 4} (1998),
377-418.
\bibitem{D}I. Dolguntseva, The Hochschild cohomolohy for associative conformal algebras, Algebra Logika {\bf 46} (2007), 688-706; English tranl., Algebra Logic {\bf 46} (2007), 373-384.

\bibitem{H1} Y. Hong, Extending structures for associative conformal algebras, Linear Multilinear Algebra {\bf 67} (2019), 196-212.
\bibitem{HB} Y. Hong, C. Bai, Conformal classical Yang-Baxter equation,
$S$-equation and $\mathcal{O}$-operators, Lett. Math. Phys. {\bf
110} (2020), 885-909.
\bibitem{HL1} Y. Hong, F. Li, Left-symmetric conformal algebras and vertex algebras, J. Pure. Appl. Algebra {\bf 219} (2015), 3543-3567.
\bibitem{HL} Y. Hong, F. Li, On left-symmetric conformal bialgebras, J. Algebra Appl.  {\bf 14} (2015), 1450079.
\bibitem{JR} S. Joni, G. Rota, Coalgebras and bialgebras in combinatorics, Stud. Appl. Math. {\bf 61} (1979), 93-139; Reprinted in Gian-Carto Rota on combinatorics: Introductory papers and commentaries, edited by J. Kung, Birkh\"auser, Boston, 1995.
\bibitem{K1} V. Kac, Vertex algebras for beginners, 2nd Edition, Amer. Math. Soc. Providence, RI, 1998.
\bibitem{K} V. Kac, The idea of locality, in Physical Applications and Mathematical Aspects of Geometry, Groups and Algebras, edited by
H.-D. Doebner et al., World Scientific Publishing, Singapore,
1997, 16-32.
\bibitem{K2} V. Kac, Formal distribution algebras and conformal
algebras, Brisbane Congress in Math. Phys., 1997.
\bibitem{KA} V. Kac, R. Alexander, Simple Jordan conformal superalgebras, J. Algebra Appl. {\bf 7} (2008), 517-533.
\bibitem{Ko1}P. Kolesnikov, Simple associative conformal algebras of linear growth, J. Algebra {\bf 295} (2006), 247-268.
\bibitem{Ko2}P. Kolesnikov, Associative conformal algebras with finite faithful representation, Adv. Math. {\bf 202} (2006), 602-637.
\bibitem{Ko3}P. Kolesnikov, On the Wedderburn principal theorem
in conformal algebras, J. Algebra Appl. {\bf 6} (2007), 119-134.
\bibitem{Ko4} P. Kolesnikov, On finite representations of conformal algebras, J. Algebra {\bf 331} (2011), 169-193.
\bibitem{L}J. Liberati, On conformal bialgebras, J. Algebra {\bf 319} (2008), 2295-2318.

\bibitem{Lod} J.-L. Loday, Dialgebras. In ``Dialgebras and
Related Operads",  Lecture Notes in Math. {\bf 1763}, Springer,
Berlin, 2001, 7-66.

\bibitem{R1}A. Retakh, Associative conformal algebras of linear growth, J. Algebra {\bf 237} (2001), 769-788.
\bibitem{R2}A. Retakh, On associative conformal algebras of linear growth II, J. Algebra {\bf 304} (2006), 543-556.
\bibitem{Ro1}M. Roitman, On embedding of Lie conformal algebras into associative conformal algebras, J. Lie theory {\bf 15} (2005), 575-588.
\bibitem{Ro2}M. Roitman, Universal enveloping conformal algebras, Selecta Math. New Ser. {\bf 6} (2000), 319-345.
\bibitem{Ro3} M. Roitman, On free conformal and vertex algebras, J. Algebra {\bf 217} (1999), 496-527.
\bibitem{Z1}E. Zelmanov, On the structure of conformal algebras, Contemp.  Math. {\bf 264} (2000), 139-156.
\bibitem{Z2}E. Zelmanov, Idempotents in conformal algebras, Proceedings of the Third International Algebra Conference, Tainan, 2002, 257-266.
\bibitem{Zhe}V. Zhelyabin, Jordan bialgebras and their connection with Lie bialgebras, Algebra Logika {\bf 36} (1997),
3-35; English tranl., Algebra Logic {\bf 36} (1997), 1-15.
\end{thebibliography}
\end{document}